\documentclass[1p,12pt,number,sort&compress,times]{elsarticle}
\usepackage{hyperref}

\usepackage{amsmath,amssymb,amsfonts,amsthm,bbm,bm}
\usepackage{enumitem}

\newcommand{\NN}{\mathbbm{N}}

\newcommand{\RR}{\mathbbm{R}}

\newcommand{\uv}{\mathbf{u}}

\DeclareMathOperator{\im}{im}

\newtheorem{theorem}{Theorem}
\newtheorem{proposition}[theorem]{Proposition}
\newtheorem{lemma}[theorem]{Lemma}
\newtheorem{corollary}[theorem]{Corollary}

\theoremstyle{definition}
\newtheorem{definition}[theorem]{Definition}
\newtheorem{example}[theorem]{Example}
\theoremstyle{remark}
\newtheorem{remark}[theorem]{Remark}
\numberwithin{equation}{section}

\begin{document}

\begin{frontmatter}

\title{Singular Initial Value
  Problems for Scalar Quasi-Linear Ordinary Differential Equations} 
\author[kassel]{Werner M. Seiler}
\ead{seiler@mathematik.uni-kassel.de}
\author[kassel]{Matthias Sei\ss}
\ead{mseiss@mathematik.uni-kassel.de}
\address[kassel]{Institut f\"ur Mathematik, Universit\"at Kassel, 34132 Kassel,
  Germany}

\begin{abstract}
  We discuss existence, non-uniqueness and regularity of one- and two-sided
  solutions of initial value problems for scalar quasi-linear ordinary
  differential equations where the initial condition corresponds to an
  impasse point of the equation.  With a differential geometric approach,
  we reduce the problem to questions in dynamical systems theory.  As an
  application, we discuss in detail second-order equations of the form
  $g(x)u''=f(x,u,u')$ with an initial condition imposed at a simple zero of
  $g$.  This generalises results by Liang and also makes them more
  transparent via our geometric approach.
\end{abstract}

\begin{keyword}
  implicit differential equations \sep singular initial value problems \sep
  Vessiot distribution \sep non-uniqueness of solutions
  \MSC[2010] 34A09 \sep 34A12 \sep 34A26
\end{keyword}

\end{frontmatter}

\section{Introduction}

In this work, we are concerned with initial value problems for scalar
\emph{implicit} ordinary differential equations
\begin{equation}\label{eq:implode}
  F(x,u,u_{1},\dots,u_{q})=0
\end{equation}
where $u_{i}$ denotes the $i$th derivative of the unknown real-valued
function $u(x)$ (for convenience we identify $u_{0}=u$).  Initial data
consist of a point
$(\bar x,\bar u,\bar{u}_{1},\dots,\bar{u}_{q})\in\RR^{q+2}$ on which $F$
vanishes.  We call such an initial value problem \emph{singular}, if the
derivative $F_{u_{q}}(\bar x,\bar u,\bar{u}_{1},\dots,\bar{u}_{q})$
vanishes implying that around our initial point \eqref{eq:implode} cannot
be solved for the highest derivative $u_{q}$ and thus that standard
existence and uniqueness theorems do not apply.

In the first part of this article (Sections
\ref{sec:geoode}--\ref{sec:prol}), we will recall the relevant structures
to study equations like \eqref{eq:implode} from a geometric point of view
and define what we mean by \emph{(geometric) singularities} of a
differential equation.  The basic idea of our approach is to associate with
\eqref{eq:implode} a vector field on a submanifold of a jet bundle such
that its integral curves correspond to (prolonged) solutions.  This idea
leads naturally to a generalised notion of solutions -- \emph{geometric
  solutions} -- which do not necessarily represent the graph of a function,
but which can be understood as a concatenation of graphs and thus are very
useful for a solution theory at singularities.

The main emphasis of this article will be on the special case of
\emph{quasi-linear} equations of the form
\begin{equation}\label{eq:qlode}
  g(x,u,u_{1},\dots,u_{q-1})u_{q}=f(x,u,u_{1},\dots,u_{q-1})\,.
\end{equation}
Here it suffices to provide an arbitrary point
$(\bar x,\bar u,\bar{u}_{1},\dots,\bar{u}_{q-1})\in\RR^{q+1}$ as initial
data.  The corresponding initial value problem is singular, if the function
$g$ vanishes at this point.  Many classical equations in mathematical
physics are of this form for $q=2$.  They are usually written in explicit
form with a rational right hand side.  For the kind of problems studied by
us, it is, however, better to use the implicit representation.  We will
always assume that \eqref{eq:qlode} is in reduced form, i.\,e.\ that the
functions $f$ and $g$ do not have a non-trivial common factor.
Furthermore, we will assume that \eqref{eq:qlode} does not admit singular
integrals which is equivalent to the overdetermined system of differential
equations $f(x,u,u_{1},\dots,u_{q-1})=g(x,u,u_{1},\dots,u_{q-1})=0$ being
inconsistent and thus not possessing any solutions.

The second part of this article (Sections
\ref{sec:impasse}--\ref{sec:propimp}) is concerned with adapting the
geometric approach outlined in the first part to quasi-linear equations.
Geometric singularities of differential equations can be understood as a
special case of the theory of singularities of smooth maps between two
manifolds \cite{agv:sing1,gg:stable}.  The main emphasis in this theory has
been on classification problems for generic implicit equations of low order
(see e.\,g.\ \cite{ld:singgen,diis:gensing}).  Since quasi-linear equations
are not generic, they are not covered by these works.  By contrast, in the
context of the theory of differential algebraic equations, essentially only
the quasi-linear case has been considered (see e.\,g.\
\cite{rr:dae,rr:das,gt:singqldae,jt:sing}), but the relation to singularity
theory has not been explored.  We will show that in the case of a
quasi-linear equation the relevant geometric structure can be projected to
a lower order (this has already been noted in \cite{wms:singbif}).  This
fact leads to new phenomena not present in general implicit equations.

Our geometric approach allows us the reduce the problem essentially to the
analysis of a stationary point of a vector field and thus to a classical
question in dynamical systems theory.  However, the situation is not so
trivial, as for $q\geq2$ the arising stationary points can never be
hyperbolic.  Only for planar vector fields a fairly extensive qualitative
theory exists of the behaviour around non-hyperbolic points (see e.\,g.\
\cite{dla:qualplan}).  We will therefore concentrate in this article on
situations where at most two-dimensional centre manifolds can arise.

As a concrete demonstration of the power of the geometric approach, we will
provide in the third and final part of the article (Section
\ref{sec:soivp}) an essentially complete analysis of second-order initial
value problems of the particular form
\begin{equation}\label{eq:liang}
  g(x)u''=f(x,u,u')\,,\qquad u(y)=c_{0}\,,\quad u'(y)=c_{1}
\end{equation}
where $y$ is a simple zero of the function $g$.  For the special case
$g(x)=x$, this problem has already been studied by Liang \cite{jfl:singivp}
with classical analytical methods.  It seems to us that it is not
straightforward to extend his results to more general functions $g$ and
that the analytic approach requires a certain amount of ingenuity to guess,
say, the right integrating factor and estimates.  Furthermore, the analytic
proofs cannot really explain why in this problem a certain dichotomy and
resonances appear.  By contrast, in our geometric point of view everything
arises in a completely natural and transparent manner and it will turn out
that the key parameters are nothing but eigenvalues of Jacobians at
stationary points.

As we aim at answering for our initial value problems the standard
analytical questions of existence, (non-)uniqueness and regularity of
solutions, our study will be within the smooth category, i.\,e.\ we assume
that the functions $F$ and $f$, $g$, respectively, are smooth and we search
primarily for smooth solutions, although it will turn out that sometimes
only solutions of lower regularity exist.  In the Conclusions we will
comment on the extension of our results to the case of equations of finite
regularity $\mathcal{C}^{r}$ which is possible.  Furthermore, we will
distinguish between one-sided solutions which are required to exist only on
intervals of the form $(a,\bar x]$ or $[\bar x, b)$ and two-sided solutions
where $\bar x$ is required to be an interior point of the existence
interval.  In the theory of explicit systems, such a distinction is rarely
made, as typical existence results like the theorem of Picard-Lindel\"of
automatically provide solutions existing on both sides of the initial
point.  However, we will see that for implicit systems there are usually
less two-sided solutions than one-sided ones.

For the special case of analytic equations, there also exists a long
tradition of applying algebraic techniques to such problems like
Newton-Puiseux polygons and similar constructions.  This goes back at least
to Fine \cite{hbf:puiseux,hbf:singsol}; more modern references are e.\,g.\
\cite{bruno:local,jc:puiseux,sf:diss}.  The main thrust of these works is
the construction of explicit solutions in form of Puiseux series and a
discussion of their convergence.  Such results are surely of great
interest, not least because of their algorithmic character.  However, they
concern only a narrower class of equations and -- more importantly -- of
solutions.  In particular, one-sided solutions cannot be described by
series, but are characteristic for certain types of singularities.  In this
article, we therefore rely exclusively on methods from geometry and
dynamical systems theory with the consequence that we obtain only
non-constructive existence results.  Combining the algebraic and geometric
approaches is an interesting task for future works.  In \cite{lrss:gsade},
where the here used notions of singularities were extended to general
systems of ordinary and partial differential equations, we demonstrated
already that such a combination can be very powerful.

This article is structured as follows.  In the next section, we briefly
recall some basics of the geometric theory of ordinary differential
equations and in particular discuss our key tool, the Vessiot spaces.
Section~\ref{sec:gensol} introduces two generalised notions of solutions:
generalised solutions are curves in a jet bundle and geometric solutions
are their projections to the base space.  Statements about the regularity
of solutions require the analysis of prolongations of the given equation.
In Section~\ref{sec:prol}, we will in particular study how singularities
behave under prolongation.  The next section specialises to quasi-linear
equations.  We will demonstrate that their analysis can be performed one
order lower which leads to the notion of impasse points.  We will show that
impasse points do not necessarily come from singularities, an observation
which entails that quasi-linear equations indeed require their own theory.
In Section~\ref{sec:weaksol}, we introduce and study weak generalised
solution as a form of generalised solutions adapted to quasi-linear
equations.  Section~\ref{sec:propimp} describes our approach of reducing
the local solution behaviour around a proper impasse point to the analysis
of a stationary point of a dynamical system.  Finally, we apply in the
following section all the developed tools to the analysis of the initial
value problem \eqref{eq:liang}.

\section{The Geometry of Ordinary Differential Equations}
\label{sec:geoode}

In the geometric theory of differential equations \cite{sau:jet,wms:invol},
(systems of) differential equations are represented by an intrinsic object,
a fibred submanifold of the appropriate jet bundle.  The $q$-jet of a
smooth function\footnote{For notational simplicity, we will use throughout
  a global notation, although all our results are of a local nature.  Thus
  strictly speaking, $\phi$ is only defined on some open subset of $\RR$
  which we, however, suppress.}  $\phi:\RR\rightarrow\RR$ at a point
$x\in\RR$ is the equivalence class $[\phi]^{(q)}_{x}$ of all smooth
functions which have at $x$ the same Taylor expansion up to order $q$ as
$\phi$ and can be identified with the corresponding Taylor polynomial.  The
$q$th order \emph{jet bundle} $\mathcal{J}_{q}=\mathcal{J}_{q}(\RR,\RR)$
consists of all such $q$-jets and defines an $(q+2)$-dimensional manifold.
We identify $\mathcal{J}_{0}=\RR^{1+1}$ with the space of the independent
variable $x$ and the dependent variable $u$.  By the theorem of Taylor,
coordinates on $\mathcal{J}_{q}$ are given by
$(x,u,u',\dots,u^{(q)})=(x,\uv_{(q)})$ where $u^{(q)}$ denotes the
derivatives of order $q$ and $\uv_{(q)}$ the collection of all derivatives
from order $0$ up to $q$.  For orders $q>r$, there are natural projection
maps $\pi^{q}_{r}:\mathcal{J}_{q}\rightarrow\mathcal{J}_{r}$ between the
corresponding jet bundles where simply the higher derivatives are
``forgotten''.  In addition, we have the projection
$\pi^{q}:\mathcal{J}_{q}\rightarrow\RR$ to the base space where everything
except the expansion point $x$ is ``forgotten''.  We define a \emph{scalar
  differential equation} as a hypersurface
$\mathcal{R}_{q}\subseteq\mathcal{J}_{q}$ such that
$\pi^{q}(\mathcal{R}_{q})$ lies dense in $\RR$.  In the classical geometric
theory, one requires that the restriction of $\pi^{q}$ to $\mathcal{R}_{q}$
defines a surjective submersion.  However, this condition excludes the
appearance of any kind of singularity.  We will therefore use our relaxed
condition which still suffices to ensure that $x$ is indeed an independent
variable.

\begin{remark}\label{rem:algsing}
  In practice, the set $\mathcal{R}_{q}$ is given as the zero set of some
  smooth function $F:\mathcal{J}_{q}\rightarrow\RR$.  Even if we assume for
  simplicity that this function is analytic, we must expect that
  $\mathcal{R}_{q}$ is not a manifold, but only an analytic variety which
  may possess singularities in the sense of analytic geometry.  We call
  such points \emph{algebraic singularities} in contrast to the
  \emph{geometric singularities} which are the topic of this article.  As
  currently not much is known about the local solution behaviour around
  algebraic singularities (see the recent works \cite{sf:diss,ss:casc20}
  for some results), we will ignore them here.  Thus, strictly speaking, we
  do not work with the whole variety $\mathcal{R}_{q}$, but always restrict
  to its smooth part (which we call again $\mathcal{R}_{q}$).  In concrete
  examples, we will ensure that we study only smooth points.
\end{remark}

A very important geometric structure on the jet bundle $\mathcal{J}_{q}$
for $q\geq1$ is provided by the \emph{contact distribution}
$\mathcal{C}^{(q)}\subset T\mathcal{J}_{q}$ which encodes geometrically the
chain rule and thus the different roles played by the various jet
variables.  In our case of scalar ordinary differential equations, it is
spanned by two vector fields: a transversal one
\begin{equation}\label{eq:Ctrans}
  C^{(q)}=\partial_{x}+u^{(1)}\partial_{u}+\cdots+
                      u^{(q)}\partial_{u^{(q-1)}}
\end{equation}
and a vertical one (with respect to the fibration $\pi^{q}$ to the base
space)
\begin{equation}\label{eq:Cvert}
  C_{q}=\partial_{u^{(q)}}\,.
\end{equation}
To avoid case distinctions, we set $C^{(0)}=\partial_{x}$ and
$C_{0}=\partial_{u}$.  By abuse of notation, we use the vector fields
$C^{(q)}$ and $C_{q}$ on any jet bundle $\mathcal{J}_{r}$ with $r\geq q$
without writing out the needed pull-back.

Given a differential equation $\mathcal{R}_{q}\subseteq\mathcal{J}_{q}$ and
a (smooth) point $\rho=(\bar x,\bar{\uv}_{(q)})\in\mathcal{R}_{q}$ on it,
that part of the contact distribution which is tangential to
$\mathcal{R}_{q}$ is the \emph{Vessiot space}
$\mathcal{V}_{\rho}[\mathcal{R}_{q}]=
T_{\rho}\mathcal{R}_{q}\cap\mathcal{C}^{(q)}|_{\rho}$ at $\rho$.  We will
see below that the elements of the Vessiot space may be interpreted as
infinitesimal solutions (or integral elements in the language of Cartan).
The family of all Vessiot spaces is called the \emph{Vessiot distribution}
$\mathcal{V}[\mathcal{R}_{q}]$ of $\mathcal{R}_{q}$, although in general
$\mathcal{V}[\mathcal{R}_{q}]$ defines a regular smooth distribution only
on an open subset of $\mathcal{R}_{q}$.

Computing the Vessiot space $\mathcal{V}_{\rho}[\mathcal{R}_{q}]$ at a
point $\rho\in\mathcal{R}_{q}$ is straightforward and requires only linear
algebra.  Any vector $X\in\mathcal{V}_{\rho}[\mathcal{R}_{q}]$ lies in the
contact distribution $\mathcal{C}^{(q)}|_{\rho}$ and thus is a linear
combination of the basic contact fields:
$X=aC^{(q)}|_{\rho}+bC_{q}|_{\rho}$.  On the other hand, $X$ must be
tangent to $\mathcal{R}_{q}$.  It is well-known that hence $X$ must satisfy
the equation $X(F)(\rho)=0$ where we again assume that $\mathcal{R}_{q}$ is
given as the zero set of the function $F:\mathcal{J}_{q}\rightarrow\RR$.
Entering our ansatz yields then the following linear equation for the two
coefficients $a$ and $b$:
\begin{equation}\label{eq:vessdist}
  C^{(q)}(F)(\rho)a+C_{q}(F)(\rho)b=0\,.
\end{equation}
Note that $X$ is vertical for $\pi^{q}$, if and only if the coefficient $a$
vanishes.  Obviously, at almost all points $\rho\in\mathcal{R}_{q}$ the
Vessiot space $\mathcal{V}_{\rho}[\mathcal{R}_{q}]$ is one-dimensional.

\section{Generalised and Geometric Solutions}
\label{sec:gensol}

From a geometric point of view, we identify any function
$\phi:\RR\rightarrow\RR$ with its graph or, more precisely, we prefer to
consider instead of the function $\phi$ the section
$\sigma_{\phi}:\RR\rightarrow\RR^{2},\ x\mapsto \bigl(x,\phi(x)\bigr)$
whose image is the graph of $\phi$.  It induces naturally a section of any
jet bundle $\mathcal{J}_{q}$ with $q\geq1$, namely the \emph{prolonged
  section}
\begin{displaymath}
  j_{q}\sigma_{\phi}\,:\,\RR\rightarrow\mathcal{J}_{q},\quad
  x\mapsto\bigl(x,\phi(x),\phi'(x),\dots,\phi^{(q)}(x)\bigr)
\end{displaymath}
Obviously, $j_{q}\sigma_{\phi}$ can be defined only at points $x$ where
$\phi$ is at least $q$ times differentiable.  A \emph{(strong) solution} of
a differential equation $\mathcal{R}_{q}\subseteq\mathcal{J}_{q}$ is a
function $\phi$ such that the image of $j_{q}\sigma_{\phi}$ lies completely
in the manifold $\mathcal{R}_{q}$.  This represents a natural geometric
formulation of the usual notion of a solution.  The Vessiot distribution
allows us to introduce a more general concept of solutions which helps in
the understanding of singularities.

\begin{definition}\label{def:gensol}
  A \emph{generalised solution} of the differential equation
  $\mathcal{R}_{q}$ is a one-dimensional integral manifold
  $\mathcal{N}\subseteq\mathcal{R}_{q}$ of the Vessiot distribution
  $\mathcal{V}[\mathcal{R}_{q}]$, i.\,e.\ at every point
  $\rho\in\mathcal{N}$ we have
  $T_{\rho}\mathcal{N}\subseteq\mathcal{V}_{\rho}[\mathcal{R}_{q}]$.  A
  generalised solution is \emph{proper}, if there does not exist a point
  $x\in\RR$ such that $\mathcal{N}\subseteq(\pi^{q})^{-1}(x)$.  The
  projection $\pi_{0}^{q}(\mathcal{N})\subset\mathcal{J}_{0}$ of a proper
  generalised solution is called a \emph{geometric solution}.
\end{definition}

If $\phi$ is a strong solution, then $\im{\sigma_{\phi}}$ is a geometric
solution coming from the generalised solution $\im{j_{q}\sigma_{\phi}}$.
However, not all geometric solutions are graphs of functions.  In fact,
they are not even necessarily smooth curves, as they arise via a
projection.  If a generalised solution is not proper, then it is of no
interest for an existence theory, as it lives completely over a single
point $x\in\RR$.  Sometimes such solutions can be useful as
separatrices.

In an \emph{initial value problem}, we prescribe a point\footnote{For
  implicit equations, it is generally necessary to prescribe also a value
  for the $q$th order derivative, as it is usually not uniquely determined
  by the differential equation.}
$\rho=(\bar x, \bar{\uv}_{(q)})\in\mathcal{R}_{q}$ and look for proper
generalised solutions containing $\rho$.  We distinguish between
\emph{one-sided} generalised solutions $\mathcal{N}$ where $\rho$ is a
boundary point of $\mathcal{N}$ and \emph{two-sided} solutions where $\rho$
is an interior point of $\mathcal{N}$.

In the case of an explicit equation, the classical existence and uniqueness
theorems for ordinary differential equations imply the existence of a
unique two-sided generalised solution through every point
$\rho_{q}\in\mathcal{R}_{q}$ and this generalised solution projects on a
strong solution (see Theorem~\ref{thm:regsing} below).  In the case of an
implicit equation, the situation becomes more involved at certain points,
namely the geometric singularities.  Following Arnold \cite{via:geoode}, we
will use the following taxonomy for smooth points on
$\mathcal{R}_{q}$.\footnote{Rabier \cite{pjr:singular} considers this
  terminology as ``inappropriate'', because the same terms appear in the
  Fuchs-Frobenius theory of linear ordinary differential equations with a
  different meaning.  However, from a geometric point of view, the
  terminology is very natural and as it has become standard in singularity
  theory, we will stick to it.}

\begin{definition}\label{def:sing}
  A smooth point $\rho\in\mathcal{R}_{q}$ is an \emph{irregular
    singularity} of the differential equation
  $\mathcal{R}_{q}\subset\mathcal{J}_{q}$, if
  $\dim{\mathcal{V}_{\rho}[\mathcal{R}_{q}]}>1$.  In the case of a
  one-dimensional Vessiot space, we further distinguish whether or not it
  lies transversal to the canonical fibration $\pi^{q}$.  If
  $\mathcal{V}_{\rho}[\mathcal{R}_{q}]$ is vertical (i.\,e.\ all solutions
  of \eqref{eq:vessdist} satisfy $a=0$), then the point $\rho$ is a
  \emph{regular singularity}.  Otherwise $\rho$ is a \emph{regular point}.
\end{definition}

It follows thus from \eqref{eq:vessdist} that $\rho\in\mathcal{R}_{q}$ is a
regular point, if and only if $C_{q}(F)(\rho)\neq0$.  A singularity is
irregular, if and only if not only $C_{q}(F)(\rho)$ vanishes but also
$C^{(q)}(F)(\rho)$.  Hence, we may conclude that generically all the
singularities form a submanifold of codimension $1$ and the irregular
singularities one of codimension $2$.

\begin{remark}\label{rem:irreg}
  It follows trivially from \eqref{eq:vessdist} that away from the
  irregular singularities the Vessiot distribution is smooth and regular.
  Thus in any simply connected domain $\Omega\subseteq\mathcal{R}_{q}$
  without irregular singularities $\mathcal{V}[\mathcal{R}_{q}]$ can be
  generated by a smooth vector field $X$.  One can show that such a field
  $X$ can be smoothly extended to any irregular singularity $\rho$ lying in
  the boundary of $\Omega$ and that generically it will vanish there
  \cite[Prop.~20]{ss:casc20}.\footnote{In an earlier version of this result
    \cite[Thm.~4.2]{wms:aims} some crucial conditions were omitted.  There
    are certain non-generic situations -- in particular when singular
    integrals exist -- where $X$ may not vanish at the irregular
    singularity.}  Now generalised solutions of $\mathcal{R}_{q}$ through
  $\rho$ are \emph{invariant} manifolds of the extended vector field $X$
  and thus we may study the local solution behaviour around $\rho$ with the
  help of dynamical systems theory.  In particular, it is now obvious that
  generally several generalised solutions will intersect at an irregular
  singularity.
\end{remark}

Away from irregular singularities, the existence and uniqueness theory of
differential equations satisfying our assumptions is rather simple.  We
recall the following result from \cite{wms:aims} which generalises the
classical existence and uniqueness theorem for explicit ordinary
differential equations.  We also include the short proof, as it makes the
underlying geometry more transparent.

\begin{theorem}\label{thm:regsing}
  Let $\mathcal{R}_{q}\subset\mathcal{J}_{q}$ be a scalar ordinary
  differential equation of order $q$ such that at every point
  $\rho\in\mathcal{R}_{q}$ the Vessiot space
  $\mathcal{V}_{\rho}[\mathcal{R}_{q}]$ is one-dimensional (i.\,e.\ there
  are no irregular singular points).  If $\rho$ is a regular point, then
  there exists a unique strong solution $\sigma$ with
  $\rho\in\im{j_q\sigma}$.  This solution is two-sided.  More precisely, it
  can be extended in both directions until $\im{j_q\sigma}$ reaches either
  the boundary of $\mathcal{R}_{q}$ or a regular singular point.  If $\rho$
  is a regular singular point, then either two strong one-sided solutions
  $\sigma_1,\sigma_2$ exist with $\rho\in\overline{\im{j_q\sigma_i}}$ which
  either both start or both end in $\rho$ or only one strong two-sided
  solution exists whose $(q+1)$th derivative blows up at $x=\pi^{q}(\rho)$.
\end{theorem}

\begin{proof}
  By the made assumptions, $\mathcal{V}[\mathcal{R}_{q}]$ is a smooth
  regular one-dimensional distribution and hence trivially involutive.  The
  Frobenius theorem guarantees for each point $\rho\in\mathcal{R}_{q}$ the
  existence of a unique generalised solution $\mathcal{N}_{\rho}$ with
  $\rho\in\mathcal{N}_{\rho}$.  This generalised solution is a smooth curve
  which can be extended until it reaches the boundary of $\mathcal{R}_{q}$
  and around each regular point $\bar{\rho}\in\mathcal{N}_{\rho}$ it
  projects onto the graph of a strong solution $\sigma$, since
  $\mathcal{V}_{\bar{\rho}}[\mathcal{R}_{q}]$ is transversal.

  Assume that in an open, simply connected neighbourhood of $\rho$ the
  Vessiot distribution $\mathcal{V}[\mathcal{R}_{q}]$ is generated by the
  vector field $X$.  If $\rho$ is a regular singular point, then $X_{\rho}$
  is vertical for $\pi^q$, i.\,e.\ its $\partial_{x}$-component vanishes.
  The behaviour of the projection
  $\widetilde{\mathcal{N}}_{\rho}=\pi^{q}_{0}(\mathcal{N}_{\rho})$ depends
  on whether or not the $\partial_{x}$-component changes its sign
  at~$\rho$.  If the sign changes, then $\widetilde{\mathcal{N}}_{\rho}$
  has two branches corresponding to two strong solutions which either both
  end or both begin at $\hat{\rho}=\pi^{q}_{0}(\rho)$.  Otherwise
  $\widetilde{\mathcal{N}}_{\rho}$ is around $\hat{\rho}$ the graph of a
  strong solution, but Remark \ref{rem:vdprol} below implies that the
  $(q+1)$th derivative of this solution at $x=\pi^{q}(\rho)$ is infinite.
\end{proof}

As already mentioned in Remark \ref{rem:irreg}, we expect that at an
irregular singularity several (possibly infinitely many) generalised
solutions meet and thus the classical uniqueness statements fail.  However,
there are situations when only a unique \emph{proper} generalised solution
goes through an irregular singularity.  In such a case, this solution is
completely regular and the singular character of the singular point lies
solely in the behaviour of the nearby generalised solutions.  Around a
regular point, the generalised solutions define a regular foliation.

\begin{example}\label{ex:appsing}
  We consider the scalar first-order equation
  $\mathcal{R}_{1}\subset\mathcal{J}_{1}$ described by $u(u')^{2}+x=0$.
  The linear equation defining the Vessiot spaces takes here the form
  $\bigl(1+(u')^{3}\bigr)a+2uu'b=0$.  Hence singularities are all those
  points on $\mathcal{R}_{1}$ where either $u=0$ or $u'=0$.  Irregular
  singularities must satisfy in addition $(u')^{3}=-1$.  Hence we have
  exactly one irregular singularity, namely the point $\rho=(0,0,-1)$.
  Outside of this point $\rho$, the Vessiot distribution
  $\mathcal{V}[\mathcal{R}_{1}]$ is one-dimensional and spanned by the
  vector field
  $X=2uu'(\partial_{x}+u'\partial_{u})-\bigl(1+(u')^{3}\bigr)\partial_{u'}$.
  
  It should be noted that -- although the explicit coordinate expression
  seems to indicate otherwise -- the vector field $X$ is defined only on
  the two-dimensional manifold $\mathcal{R}_{1}$ and not on the whole jet
  bundle $\mathcal{J}_{1}$.  In this particular case, it is obvious that
  $u$ and $u'$ could be used as parameters, as $\mathcal{R}_{1}$ is the
  graph of a function $x=h(u,u')$.  In general, it is hard (if not
  impossible) to find a global parametrisation.  Therefore, we will work
  throughout this article with the redundant coordinates of the ambient jet
  bundle.
  
  Obviously, at the point $\rho$ the vector field $X$ vanishes.  The
  Jacobian of $X$ evaluated at the singularity $\rho$ is given by the
  matrix
  \begin{displaymath}
    \begin{pmatrix}
      0 & -2 & 0\\
      0 & 2 & 0\\
      0 & 0 & -3
    \end{pmatrix}\,.
  \end{displaymath}
  It has the eigenvalues $0$, $2$ and $-3$.  The eigenvector to the first
  eigenvalue, $(1,0,0)^{T}$, is not tangential to $\mathcal{R}_{1}$ and
  hence irrelevant.\footnote{The appearance of this spurious
    eigenvalue/vector is a consequence of our use of redundant coordinates.
    If we had worked with a proper parametrisation of $\mathcal{R}_{1}$, we
    would have obtained a $2\times2$ Jacobian and no spurious eigenvalue
    could have arisen.  It is, however, generally much easier to check an
    eigenvector for tangency than to find good parametrisations.}  The
  eigenvector to the third eigenvalue, $(0,0,1)^{T}$, is tangent to the
  fibre $(\pi^{1}_{0})^{-1}(0,0)$ and it is easy to see that this fibre
  actually represents the corresponding invariant manifold: it is
  completely contained in $\mathcal{R}_{1}$ and the vector field $X$ is
  vertical at every point on it.  Moreover, the fibre consists entirely of
  regular singularities.  Thus we do not get a proper generalised solution
  out of this eigenvalue.  The invariant manifold corresponding to the
  second eigenvalue (which is tangent to the eigenvector $(1,-1,0)^{T}$) is
  the unique proper generalised solution through $\rho$.  Since the tangent
  vector is transversal with respect to the fibration $\pi^{1}$ (its first
  component does not vanish), the corresponding geometric solution is a
  strong solution, namely $u(x)=-x$.

  It is not difficult to compute the general solution of this implicit
  equation via separation of variables: it is given by
  \begin{equation}\label{eq:sol1}
    u(x)=
    \begin{cases}
      \sqrt[3]{(C\pm\sqrt{-x^{3}})^{2}} & x\leq0\\
      -\sqrt[3]{(C\pm\sqrt{x^{3}})^{2}} & x\geq0
    \end{cases}
  \end{equation}
  which yields for $C=0$ the above mentioned strong solution whose
  prolongation passes through the irregular singularity.  For all other
  values of the parameter $C$, the corresponding generalised solutions hit
  the regular singularities $(0,\pm\sqrt[3]{C^{2}},0)$ depending on whether
  we approach $x=0$ from the left or from the right.  In the first case we
  always find two solutions ending in this point, in the second case two
  solutions start in this point.  The two solutions are always obtained by
  the two different signs in the corresponding branch of \eqref{eq:sol1}.
  Thus we are here in the generic case of Theorem \ref{thm:regsing}.  Note
  that the unique generalised solution through the irregular singularity is
  the only generalised solution existing for positive and for negative
  values of $x$.
\end{example}

\section{Prolongations}\label{sec:prol}

By the definition of a jet bundle, the manifold
$\mathcal{R}_{q}\subseteq\mathcal{J}_{q}$ contains only information about
derivatives up to order $q$.  For a regularity theory, we must also be able
to speak about higher-order derivatives.  This means that we must also look
at the prolongations of $\mathcal{R}_{q}$.  An intrinsic geometric
description of the prolongation process is somewhat cumbersome
\cite{wms:invol}, but in local coordinates it becomes straightforward.
Assume that the differential equation
$\mathcal{R}_{q}\subseteq\mathcal{J}_{q}$ is given as the zero set of a
function $F:\mathcal{J}_{q}\rightarrow\RR$. Then the \emph{first
  prolongation} $\mathcal{R}_{q+1}\subseteq\mathcal{J}_{q+1}$ is the zero
set of both the function $F$ and its formal derivative
\begin{equation}\label{eq:fderiv}
  D_{x}F=C^{(q)}(F)+C_{q}(F)u^{(q+1)}:
  \mathcal{J}_{q+1}\rightarrow\RR\,.
\end{equation}
$\mathcal{R}_{q+1}$ is not necessarily a manifold anymore, but for
simplicity we will assume in the sequel that it is.  Iteration of this
process yields the \emph{higher prolongations}
$\mathcal{R}_{q+r}\subseteq\mathcal{J}_{q+r}$ for any $r\in\NN$: at each
prolongation order $r$ we have to add one further equation
$D_{x}^{r}F(x,\uv_{(q+r)})=0$.  Note that the formal derivative always
yields a quasi-linear function, since we have
$\partial(D_{x}F)/\partial u^{(q+1)}= C_{q}(F)=\partial F/\partial
u^{(q)}$.

\begin{remark}\label{rem:vdprol}
  Prolonging the differential equation $\mathcal{R}_{q}$ requires
  essentially the same computations as determining its Vessiot spaces
  $\mathcal{V}_{\rho}[\mathcal{R}_{q}]$.  Indeed, we may consider
  (\ref{eq:vessdist}) as a homogenised (or ``projective'') form of
  (\ref{eq:fderiv}) considered as a linear equation for $u^{(q+1)}$.  For
  any solution $(a,b)$ of (\ref{eq:vessdist}) with $a\neq0$, we may
  identify the quotient $b/a$ with the coordinate $u^{(q+1)}$ of a point
  $\hat{\rho}\in\mathcal{R}_{q+1}\cap(\pi^{q+1}_{q})^{-1}(\rho)$ and
  conversely any such point defines a one-parameter family of solutions
  $(a,b)$ of (\ref{eq:vessdist}).  If we assume that (\ref{eq:vessdist})
  has no solution $(a,b)$ with $a\neq0$, i.\,e.\ the Vessiot space
  $\mathcal{V}_{\rho}[\mathcal{R}_{q}]$ is vertical, then by the same
  reasoning there cannot exist a point
  $\hat{\rho}\in\mathcal{R}_{q+1}\cap(\pi^{q+1}_{q})^{-1}(\rho)$ which
  implies that any strong solution $\phi$ with
  $\rho\in\im{j_{q}\sigma_{\phi}}$ lies in
  $\mathcal{C}^{q}\setminus\mathcal{C}^{q+1}$, i.\,e.\ is of finite
  regularity.
\end{remark}

This observation has the following implications for solutions of a
differential equation $\mathcal{R}_{q}\subset\mathcal{J}_{q}$.  Assume that
we have a proper generalised solution
$\mathcal{N}_{r}\subseteq\mathcal{R}_{r}$ living on some prolongation of
order $r\geq q$ which projects on a strong solution, i.\,e.\ the
corresponding geometric solution is the graph of a function.  Then this
function is -- by definition of the jet bundle -- at least of class
$\mathcal{C}^{r}$ and all projections $\pi^{r}_{r'}(\mathcal{N}_{r})$ to an
order $q\leq r'\leq r$ define generalised solutions of the corresponding
prolongations $\mathcal{R}_{r'}$.  However, a generalised solution
$\mathcal{N}_{r+1}$ of the next prolongation $\mathcal{R}_{r+1}$ projecting
onto $\mathcal{N}_{r}$ will only exist, if this function is at least of
class $\mathcal{C}^{r+1}$.

Hence we can make statements about the regularity of solutions by studying
the behaviour of the prolongations.  As a first step, we translate the
observation made in Remark~\ref{rem:vdprol} into a statement about the
fibres above points on $\mathcal{R}_{q}$ by combining it with
Definition~\ref{def:sing}.

\begin{proposition}\label{prop:fibre}
  Let $\rho_{q}\in\mathcal{R}_{q}$ be an arbitrary point on the $q$th order
  differential equation $\mathcal{R}_{q}$ and consider the fibre
  $\mathcal{F}_{q+1}=(\pi^{q+1}_{q})^{-1}(\rho_{q})\cap\mathcal{R}_{q+1}$
  above it in the first prolongation $\mathcal{R}_{q+1}$.  If $\rho_{q}$ is
  a regular point, then $\mathcal{F}_{q+1}$ is non-empty and consists
  entirely of regular points of $\mathcal{R}_{q+1}$.  If $\rho_{q}$ is a
  regular singularity, then $\mathcal{F}_{q+1}$ is empty.  In the case of
  an irregular singularity, the fibre $\mathcal{F}_{q+1}$ is non-empty and
  consists entirely of singular points of $\mathcal{R}_{q+1}$.
\end{proposition}

\begin{proof}
  For regular singularities, the assertion follows immediately from
  Remark~\ref{rem:vdprol}, as for them the Vessiot space is vertical by
  definition.  The remark also implies that in the other two cases, the
  fibre is non-empty.  The respective statement about the nature of the
  points in the fibre follows from the quasi-linearity of the formal
  derivative.  Singular points are characterised by the vanishing of the
  Jacobian with respect to the highest order derivative and hence the above
  mentioned equality
  $\partial F/\partial u^{(q)}=\partial(D_{x}F)/\partial u^{(q+1)}$ entails
  the claim.
\end{proof}

Note that in the case of an irregular singularity, we only assert that the
fibre consists of singular points -- nothing is said about whether these
are regular or irregular.  As we will see later, essentially everything is
possible.  One consequence of Proposition \ref{prop:fibre} is that there
can never exist a smooth (generalised) solution through a regular singular
point (we saw this already in Theorem \ref{thm:regsing}).  The regularity
of any geometric solution reaching a regular singularity
$\rho\in\mathcal{R}_{q}$ is there always $q$.

A necessary condition for the existence of a generalised solution through
an irregular singular point $\rho_{q}\in\mathcal{R}_{q}$ projecting on a
geometric solution which is the graph of a function of regularity
$\mathcal{C}^{r}$ for some $r\geq q$ is that the fibre above it contains at
least one irregular singularity $\rho_{q+1}\in\mathcal{R}_{q+1}$ and that
the same holds for the fibre $\mathcal{F}_{q+2}$ above $\rho_{q+1}$ and so
on until prolongation order $r$.  If some fibre contains more than one
irregular singularity, then it is possible that several such solutions go
through $\rho_{q}$.  In the case of a fibre consisting entirely of
irregular singularities, this may even be infinitely many.  In Remark
\ref{rem:irreg}, we mentioned that we can study the generalised solutions
around the irregular singularity $\rho_{q}\in\mathcal{R}_{q}$ by analysing
the local phase portrait of a vector field $X$.  If there are local
solutions of different regularity, then the local phase portraits around
the sequence of irregular singularities $\rho_{q}$, $\rho_{q+1}$, \dots\
may qualitatively change at some order.  This observation will allow us
statements about the regularity of solutions.

\section{Impasse Points of Quasi-Linear Equations}
\label{sec:impasse}

In the previous sections, we considered arbitrary implicit ordinary
differential equations.  From now on, we specialise to quasi-linear
equations of the form \eqref{eq:qlode}, i.\,e.\
$g(x,\uv_{(q-1)})u^{(q)}=f(x,\uv_{(q-1)})$.  From a geometric point of
view, quasi-linearity means that $\mathcal{R}_{q}$ is an affine subbundle
of $\mathcal{J}_{q}$ (see the discussion in \cite[Rem.~10.1.4]{wms:invol}).
The linear equation \eqref{eq:vessdist} determining the Vessiot space at a
point $\rho=(\bar x,\bar{\uv}_{(q)})\in\mathcal{R}_{q}$ takes then the form
\begin{equation}\label{eq:vdql}
  \Bigl[C^{(q)}(g)(\rho)\bar{u}^{(q)}-C^{(q)}(f)(\rho)\Bigr]a+g(\rho)b=0\,.
\end{equation}
Whether or not the point $\rho$ is a singularity is independent of the
value of $\bar{u}^{(q)}$, as all singularities are obviously characterised
by the condition $g(\rho)=0$ and by assumption the function $g$ does not
depend on $u^{(q)}$.  A singularity is irregular, if and only if in
addition the equation
\begin{equation}\label{eq:irreg}
  C^{(q)}(g)(\rho)\bar{u}^{(q)}-C^{(q)}(f)(\rho)=0
\end{equation}
holds and this condition generally depends on the value of $\bar{u}^{(q)}$.

The key property of a quasi-linear equation
$\mathcal{R}_{q}\subset\mathcal{J}_{q}$ is that its analysis can be
performed already in the jet bundle $\mathcal{J}_{q-1}$ of one order less
(see also the discussion in \cite{wms:singbif}).  Consider the subset
$\widetilde{\mathcal{R}}_{q-1}=
\pi^{q}_{q-1}(\mathcal{R}_{q})\subseteq\mathcal{J}_{q-1}$ obtained by
projecting the $q$the order differential equation $\mathcal{R}_{q}$ into
the jet bundle $\mathcal{J}_{q-1}$ of one order less.  A point
$\tilde{\rho}\in\mathcal{J}_{q-1}$ lies in $\widetilde{\mathcal{R}}_{q-1}$,
if and only if either $g(\tilde{\rho})\neq0$ (in this case the fibre
$\mathcal{F}_{q}=(\pi^{q}_{q-1})^{-1}(\tilde{\rho})\cap\mathcal{R}_{q}$
consists of exactly one point $\rho$ which is regular for
$\mathcal{R}_{q}$) or $g(\tilde{\rho})=f(\tilde{\rho})=0$ (now the fibre
$\mathcal{F}_{q}$ is one-dimensional, i.\,e.\ it contains infinitely many
points which are all singularities).

On the open subset $\mathcal{S}_{q}\subseteq\mathcal{R}_{q}$ obtained by
removing all irregular singularities of the differential equation, the
Vessiot spaces form a one-dimensional smooth distribution
$\mathcal{V}[\mathcal{R}_{q}]$ generated by the vector field
\begin{equation}\label{eq:X}
  X=gC^{(q)}-\Bigl[C^{(q)}(g)u^{(q)}-C^{(q)}(f)\Bigr]C_{q}\,.
\end{equation}
This follows immediately from solving \eqref{eq:vdql}.  Writing out
$C^{(q)}$ and noting that everywhere on $\mathcal{R}_{q}$ we have
$gu^{(q)}=f$, we find that the vector field $X$ is projectable along
$\pi^{q}_{q-1}$ to the subset
$\widetilde{\mathcal{S}}_{q-1}= \pi^{q}_{q-1}(\mathcal{S}_{q}) \subseteq
\widetilde{\mathcal{R}}_{q-1}\subseteq\mathcal{J}_{q-1}$.  In other words,
if we take for each point $\rho\in\mathcal{S}_{q}$ the vector
$X_{\rho}\in T_{\rho}\mathcal{R}_{q}$ obtained by evaluating the vector
field \eqref{eq:X} at $\rho$ and then project it with the tangent map
$T_{\rho}\pi^{q}_{q-1}$ into the tangent space
$T_{\tilde{\rho}}\widetilde{\mathcal{R}}_{q-1}$ at the point
$\tilde{\rho}=\pi^{q}_{q-1}$, then we obtain on
$\widetilde{\mathcal{S}}_{q-1}$ a well-defined vector field\footnote{In an
  arbitrary implicit equation typically several points
  $\rho\in\mathcal{R}_{q}$ project on the same point
  $\tilde{\rho}\in\widetilde{\mathcal{R}}_{q-1}$, but there is no reason
  why the corresponding vectors $X_{\rho}$ should be mapped on the same
  vector.  However, in the case of a quasi-linear equation, there exists
  only one point $\rho\in\mathcal{R}_{q}$ over every point
  $\tilde{\rho}\in\widetilde{\mathcal{S}}_{q-1}$ and hence we obtain a
  unique vector.} given by
\begin{equation}\label{eq:projvd}
  Y=g\Bigl(\partial_{x}+u^{(1)}\partial_{u}+\cdots+
           u^{(q-1)}\partial_{u^{(q-2)}}\Bigr)
   +f\partial_{u^{(q-1)}}=gC^{(q-1)}+fC_{q-1}\,.
\end{equation}

The thus constructed vector field $Y$ can trivially be analytically
continued to any point $\tilde{\rho}\in\mathcal{J}_{q-1}$ where both
functions $f$ and $g$ are defined.  We will assume in the sequel for
simplicity that this is the case on the whole jet bundle
$\mathcal{J}_{q-1}$, as in many applications $f$ and $g$ are polynomials
and thus indeed everywhere defined.  We continue to call this field $Y$ and
note for later use that by construction it is a contact vector field,
i.\,e.\ it lies in the contact distribution $\mathcal{C}^{(q-1)}$ on
$\mathcal{J}_{q-1}$.

\begin{remark}
  Any other vector field $\tilde{X}$ that also generates the distribution
  $\mathcal{V}[\mathcal{R}_{q}]$ is of the form $\tilde{X}=hX$ with a
  nowhere vanishing function $h$.  It will also be projectable provided $h$
  does not depend on $u^{(q)}$ and in this case its projection is simply
  $hY$.  Thus we actually obtain a whole projected distribution. Because of
  our assumption that $f$ and $g$ have no non-trivial common factor, $Y$
  may be considered as a ``minimal'' generator of this distribution without
  spurious zeros.  Therefore we will work in the sequel exclusively with
  $Y$.
\end{remark} 

\begin{definition}\label{def:impasse}
  A point $\tilde{\rho}\in\mathcal{J}_{q-1}$ is an \emph{impasse
    point}\footnote{We use the word ``impasse point'' to clearly
    distinguish from ``singular points'' which for us always live in
    $\mathcal{R}_{q}$, i.\,e.\ one order higher.  In the literature, the
    name ``impasse point'' appears mainly in the context of differential
    algebraic equations where almost exclusively quasi-linear systems are
    studied.  But it seems that every author has here his/her own
    terminology\dots} for the quasi-linear differential equation
  $\mathcal{R}_{q}\subset\mathcal{J}_{q}$, if the vector field $Y$ given by
  \eqref{eq:projvd} is not transversal to the fibration $\pi^{q-1}$ at
  $\tilde{\rho}$ (i.\,e.\ if its $\partial_{x}$-component vanishes at
  $\tilde{\rho}$).  Otherwise it is a \emph{regular point}.  An impasse
  point is \emph{proper}, if the field $Y$ vanishes at $\tilde{\rho}$.
  Otherwise it is \emph{improper}.
\end{definition}

\begin{remark}
  This possibility to first project the vector field $X$ defined on an open
  subset of the original differential equation $\mathcal{R}_{q}$ obtaining
  the vector field $Y$ defined on some subset of
  $\widetilde{\mathcal{R}}_{q-1}$ and then to continue analytically $Y$ to
  all of $\mathcal{J}_{q-1}$ is specific to quasi-linear equations and has
  no analogon for fully non-linear equations.  In particular, for fully
  non-linear equations it is not possible to study solutions outside of
  projections of $\mathcal{R}_{q}$, i.\,e.\ a notion like an improper
  impasse point cannot be introduced in classical singularity theory.
  Indeed, there cannot be a (prolonged) strong solution through an improper
  impasse point, but we will see below that it makes sense to study initial
  value problems at such points.
\end{remark}

Obviously, impasse points are characterised by the condition
$g(\tilde{\rho})=0$ and at proper impasse points we find in addition that
also $f(\tilde{\rho})=0$ which is equivalent to
$\tilde{\rho}\in\widetilde{\mathcal{R}}_{q-1}$.  We now study the
relationship of impasse points and singularities.  Using the splitting
$C^{(q)}=C^{(q-1)}+u^{(q)}C_{q-1}$, the irregularity condition
\eqref{eq:irreg} can be written as a quadratic equation for the coordinate
$\bar{u}^{(q)}$:
\begin{equation}\label{eq:irreg2}
  C^{(q-1)}(f)(\tilde{\rho}) +
  \Bigl[C_{q-1}(f)(\tilde{\rho}) -
        C^{(q-1)}(g)(\tilde{\rho})\Bigr]\bar{u}^{(q)} -  
  C_{q-1}(g)(\tilde{\rho})\bigl(\bar{u}^{(q)}\bigr)^{2}=0\,.
\end{equation}
Together with the above made observations on
$\widetilde{\mathcal{R}}_{q-1}$, the various cases for this equation lead
immediately to the following result.

\begin{proposition}\label{prop:impasse}
  Let $\tilde{\rho}\in\mathcal{J}_{q-1}$ be an arbitrary point.  If
  $\tilde{\rho}$ is regular, then there exists a unique point
  $\rho\in\mathcal{R}_{q}$ with $\pi^{q}_{q-1}(\rho)=\tilde{\rho}$ and this
  point is regular, too.  If $\tilde{\rho}$ is an improper impasse point,
  then there exists no point $\rho\in\mathcal{R}_{q}$ with
  $\pi^{q}_{q-1}(\rho)=\tilde{\rho}$.  If $\tilde{\rho}$ is a proper
  impasse point, then every point $\rho\in\mathcal{J}_{q}$ with
  $\pi^{q}_{q-1}(\rho)=\tilde{\rho}$ lies in $\mathcal{R}_{q}$ and is a
  singularity.  Four different cases arise:
  \begin{enumerate}[label=(\roman*)]
  \item If $C_{q-1}(g)(\tilde{\rho})\neq0$, then the fibre above
    $\tilde{\rho}$ contains either two or no irregular singularities
    (counted with multiplicity).  We find two irregular singularities, if
    and only if in addition
    \begin{equation}\label{eq:disc}
      \Bigl[C_{q-1}(f)(\tilde{\rho})-C^{(q-1)}(g)(\tilde{\rho})\Bigr]^{2}
      +4C^{(q-1)}(f)(\tilde{\rho})C_{q-1}(g)(\tilde{\rho})\geq 0\,.
    \end{equation}
  \item If $C_{q-1}(g)(\tilde{\rho})=0$ and $C_{q-1}(f)(\tilde{\rho})\neq
    C^{(q-1)}(g)(\tilde{\rho})$, the fibre contains a unique irregular
    singularity.
  \item If $C_{q-1}(g)(\tilde{\rho})=0$ and
    $C_{q-1}(f)(\tilde{\rho})= C^{(q-1)}(g)(\tilde{\rho})$ and
    $C^{(q-1)}(f)(\tilde{\rho})\neq0$, then there are no irregular
    singularities in the fibre.
  \item If $C_{q-1}(g)(\tilde{\rho})=0$ and
    $C_{q-1}(f)(\tilde{\rho})= C^{(q-1)}(g)(\tilde{\rho})$ and
    $C^{(q-1)}(f)(\tilde{\rho})=0$, then the entire fibre consists only of
    irregular singularities.
  \end{enumerate}
\end{proposition}

Obviously, the first case is the generic one and the non-generic cases are
above ordered by their codimension.  We see that proper impasse points
always arise below irregular singularities of $\mathcal{R}_{q}$.  But the
first and the third case show that proper impasse points can exist without
the presence of an irregular singularity.  In the generic first case, this
happens when the solutions of \eqref{eq:irreg2} are complex.  In such a
situation a classical nonlinear singularity analysis, as e.\,g.\ described
in \cite{via:geoode}, would yield nothing.  We are then dealing with a
truly quasi-linear phenomenon requiring a special analysis based on the
vector field $Y$.  Concrete instances will be studied below in Example
\ref{ex:foeq}.

\section{Weak Generalised Solutions}
\label{sec:weaksol}

Instead of studying the original equation \eqref{eq:qlode} in
$\mathcal{J}_{q}$, we may analyse the projected vector field $Y$ living on
$\mathcal{J}_{q-1}$ which means that we have transformed a non-autonomous
implicit problem into an autonomous explicit one.  This idea furthermore
leads naturally to weaker notions of solutions, as we now no longer have to
require the existence of derivatives of order $q$.

\begin{definition}\label{def:geosol}
  A \emph{weak generalised solution} of the quasi-linear differential
  equation \eqref{eq:qlode} is a one-dimensional invariant manifold
  $\mathcal{N}\subset\mathcal{J}_{q-1}$ of the vector field~$Y$, i.\,e.\ we
  have at every point $\rho\in\mathcal{N}$ that
  $Y_{\rho}\in T_{\rho}\mathcal{N}$.  A weak generalised solution
  $\mathcal{N}$ is \emph{proper}, if in addition
  $T\mathcal{N}\subseteq\mathcal{C}^{(q-1)}$ and there does not exist a
  point $x\in\RR$ such that $\mathcal{N}\subseteq(\pi^{q-1})^{-1}(x)$.  The
  projection $\pi^{q-1}_{0}(\mathcal{N})\subset\mathcal{J}_{0}$ of a proper
  weak generalised solution is a \emph{weak geometric solution}.
\end{definition}

\begin{remark}
  Note the difference in the definitions of generalised and weak
  generalised solutions: in Definition \ref{def:gensol} we used
  \emph{integral} manifolds of the Vessiot distribution
  $\mathcal{V}[\mathcal{R}_{q}]$, whereas Definition \ref{def:geosol} is
  based on \emph{invariant} manifolds of the projected vector field $Y$.
  This difference is due to the opposite behaviour of the Vessiot
  distribution and its projection at singularities and impasse points,
  respectively.  If $\rho_{q}\in\mathcal{R}_{q}$ is an irregular
  singularity, then $\dim{\mathcal{V}_{\rho_{q}}[\mathcal{R}_{q}]}$ jumps
  to an higher value.  By contrast, a proper impasse point $\rho_{q-1}$ is
  a stationary point of the vector field~$Y$ and hence the dimension of the
  projected Vessiot distribution jumps there to a lower value.  Our
  definitions are always formulated in such a way that it is possible to
  have (weak) generalised solutions going through the singularity or the
  impasse point, respectively.
\end{remark}

Again, only proper weak generalised solutions are of interest for the
analysis of initial value problems.  It is obvious that we have to exclude
generalised solutions lying completely in the fibre above a point
$x\in\RR$.  Superficially seen, it may seem as if the first condition for a
proper weak generalised solution was always automatically satisfied, as the
vector field $Y$ is constructed with the help of the Vessiot spaces and
thus of contact vector fields.  However, if $\mathcal{N}$ consists entirely
of stationary points of $Y$, then it is trivially an invariant manifold,
but there is no reason why its tangent spaces should lie in the contact
distribution.  Finding invariant manifolds at a stationary point of an
autonomous vector field represents a standard task in dynamical systems
theory.  In particular, we can apply the (un)stable and the centre manifold
theorem, respectively (see e.\,g.\ \cite{lp:deds}).

\begin{example}\label{ex:appimp}
  Like for an irregular singularity, it is possible that only a unique
  proper weak generalised solution passes through an impasse point.  As a
  concrete example, we consider an autonomous quasi-linear second-order
  equation of the form $g(u)u''=f(u')$ where we assume the existence of
  values $\bar{u}$ and $\bar{u}'$ such that $g(\bar{u})=f(\bar{u}')=0$ and
  $\bar{u}'g'(\bar{u})f'(\bar{u}')<0$ (this implies in particular that we
  have simple zeros of $g$ and $f$ respectively).  Projection of the
  Vessiot distribution yields the vector field
  $Y=g(u)\partial_{x}+u'g(u)\partial_{u}+f(u')\partial_{u'}$.  For
  arbitrary values of $\bar{x}$, the points $(\bar{x},\bar{u},\bar{u}')$
  are stationary points of $Y$ and thus proper impasse points.  The
  Jacobian of the field $Y$ at any of these points has three simple
  eigenvalues: $0$, $\bar{u}'g'(\bar{u})$ and $f'(\bar{u}')$.  By our
  assumptions, the latter two eigenvalues possess opposite signs.  The
  centre manifold consists entirely of these stationary points (and is thus
  unique \cite[Cor.~3.3]{js:cm}).  At any of them, the tangent space of the
  centre manifold is trivially spanned by the vector $\partial_{x}$.
  However, our assumptions imply that $\bar{u}'\neq0$ and thus this vector
  does not lie in the contact distribution, as it only contains the
  transversal vector $\partial_{x}+\bar{u}'\partial_{u}$.  Consequently,
  the centre manifold does not define a proper weak generalised solution.
  As at any point with $u$-coordinate equal to $\bar{u}$ the vector field
  $Y$ becomes vertical, it is easy to see that the invariant manifold
  corresponding to the last eigenvalue is simply the fibre
  $(\pi^{1}_{0})^{-1}(\bar{x},\bar{u})$ and hence also does not define a
  proper weak generalised solution.  Only the invariant manifold
  corresponding to the second eigenvalue leads to a proper weak generalised
  solution and one can show that the corresponding weak geometric solution
  is actually a strong solution.  Again the singular behaviour appears in
  the relation to the neighbouring generalised solutions.  If one considers
  the vector field $Y$ restricted to the $u$-$u'$ plane, then the point
  $(\bar{u},\bar{u}')$ is a saddle point so that the generalised solutions
  cannot define a regular foliation.
\end{example}

For the remainder of this section, we assume that the functions $f$ and $g$
are not only smooth, but analytic. Then a one-dimensional invariant
manifold $\mathcal{N}$ consists either entirely of stationary points or the
stationary points lie discrete and the points between two neighbouring
stationary points form an integral curve of the vector field $Y$ which is
an analytic manifold.  However, if a weak generalised solution
$\mathcal{N}$ consists of several integral curves, then it is not analytic
at the connecting points.  More precisely, we obtain the following result.

\begin{proposition}\label{prop:geosol}
  A weak geometric solution of an analytic differential equation
  \eqref{eq:qlode} which comes from a weak generalised solution containing
  only a discrete set of stationary points is composed of graphs of
  functions.  If $(\alpha,\omega)$ is the maximal open interval of
  definition of such a function $\phi$ and either $\alpha$ or $\omega$ is
  finite, then the function $\phi$ can be continued to $\alpha$ or
  $\omega$, respectively, and it is there $q-1$ times continuously
  differentiable, but its $q$th derivative blows up.
\end{proposition}

\begin{proof}
  Let $\mathcal{N}\subset\mathcal{J}_{q-1}$ be a proper weak generalised
  solution of \eqref{eq:qlode} which does not consists entirely of
  stationary points and
  $\widetilde{\mathcal{N}}=\pi^{q-1}_{0}(\mathcal{N})$ the corresponding
  weak geometric solution.  The form of $\widetilde{\mathcal{N}}$ depends
  on the behaviour of the restriction of the function $g$ to the curve
  $\mathcal{N}$.  We first note that $g$ cannot vanish identically on
  $\mathcal{N}$, as then $Y$ was everywhere on $\mathcal{N}$ vertical with
  respect to the fibration $\pi^{q-1}$ and $\mathcal{N}$ was not a proper
  weak generalised solution.

  If $\mathcal{C}\subseteq\mathcal{N}$ is a connected open subset on which
  $g$ vanishes nowhere, then the field $Y$ is everywhere on $\mathcal{C}$
  transversal to the fibration $\pi^{q-1}$ and we may use the independent
  variable $x$ as parameter for the curve $\mathcal{C}$.  This implies that
  $\mathcal{C}$ is the graph of a section
  $(\alpha,\omega)\rightarrow\mathcal{J}_{q-1}$.  Since the tangent field
  $Y$ of $\mathcal{C}$ is a contact field, it follows from the properties
  of the contact distribution (see e.\,g.\ \cite[Prop.~2.1.6]{wms:invol})
  that this section is the prolongation of the section associated to a
  function $\phi:(\alpha,\omega)\rightarrow\RR$.  Hence the corresponding
  piece $\widetilde{\mathcal{C}}$ of the weak geometric solution
  $\widetilde{\mathcal{N}}$ is the graph of this function $\phi$.

  If $\mathcal{C}\subseteq\mathcal{N}$ is a connected open subset of the
  curve $\mathcal{N}$ which is an integral curve of $Y$, then $\mathcal{C}$
  possesses an analytic parametrisation and we may consider the restriction
  $\hat{g}=g|_{\mathcal{C}}$ as a univariate analytic function of the curve
  parameter.  Thus $\hat{g}$ either vanishes identically or it has only
  isolated zeros.  Let $s$ be such an isolated zero and $\overline{x}$ the
  $x$-coordinate of the corresponding point $\rho\in\mathcal{N}$.
  Alternatively, let $\overline{x}$ be the $x$-coordinate of a stationary
  point of the field $Y$ which is an $\omega$ limit point of the integral
  curve $\mathcal{C}$.

  For a sufficiently close value $\underline{x}<\overline{x}$, there exists
  a function $\phi:(\underline{x},\overline{x})\rightarrow\RR$ such that
  its graph defines a piece of the weak geometric solution
  $\widetilde{\mathcal{N}}$.  We study now what happens in the limit
  $x\rightarrow\overline{x}$.  For the derivatives of $\phi$ up to order
  $q-1$, it follows immediately from the continuity of the weak generalised
  solution $\mathcal{N}$ that they converge against the value of the
  corresponding coordinate of $\rho$.  By assumption, the vector field $Y$
  is vertical at the point $\rho$.  Hence, according to Remark
  \ref{rem:vdprol}, the $q$th derivative of $\phi$ blows up at
  $\overline{x}$.  Of course, the same happens, if we consider a function
  defined to the right of the point $\overline{x}$ or a stationary point
  which is an $\alpha$ limit point of the curve $\mathcal{C}$.
\end{proof}

An \emph{initial value problem} for the quasi-linear equation
\eqref{eq:qlode} consists of prescribing a point
$\tilde{\rho}\in\mathcal{J}_{q-1}$ and asking for proper weak generalised
solutions going through it.  Improper impasse points show then a similar
behaviour as regular singularities for fully non-linear differential
equations.  Therefore we obtain in close analogy to Theorem
\ref{thm:regsing} the following existence and (non-)uniqueness theorem.

\begin{theorem}\label{thm:exnuni}
  Let $\mathcal{U}\subseteq\mathcal{J}_{q-1}$ be an open domain without
  proper impasse points of the quasi-linear differential equation
  $\mathcal{R}_{q}\subset\mathcal{J}_{q}$.  Then there exists through every
  point $\tilde{\rho}\in\mathcal{U}$ a unique weak generalised solution
  $\mathcal{N}_{\tilde{\rho}}$ of $\mathcal{R}_{q}$.  If $\tilde{\rho}$ is
  a regular point, then the corresponding weak geometric solution
  $\widetilde{\mathcal{N}}_{\tilde{\rho}}=
  \pi^{q-1}_{0}(\mathcal{N}_{\tilde{\rho}})$ is a strong solution (in some
  neighbourhood of $\bar{x}=\pi^{q-1}(\tilde{\rho})$).  If $\tilde{\rho}$
  is an improper impasse point, there are three possibilities depending on
  the behaviour of the function $g$ along the curve
  $\mathcal{N}_{\tilde{\rho}}$:
  \begin{enumerate}[label=(\roman*)]
  \item If the restricted function
    $\hat{g}=g|_{\mathcal{N}_{\tilde{\rho}}}$ vanishes identically on the
    curve $\mathcal{N}_{\tilde{\rho}}$, then $\mathcal{N}_{\tilde{\rho}}$
    is a vertical line and thus is not a proper weak generalised solution.
    Hence no weak geometric solution exists.
  \item If $\hat{g}$ does not change its sign when passing through
    $\tilde{\rho}$, then the weak geometric solution
    $\widetilde{\mathcal{N}}_{\tilde{\rho}}$ is the graph of a function
    $\phi$ which is at the point $\bar{x}$ only $q-1$ times continuously
    differentiable, as its $q$th derivative blows up.
  \item If the sign of $\hat{g}$ changes when passing through
    $\tilde{\rho}$, then $\widetilde{\mathcal{N}}_{\tilde{\rho}}$ consists
    locally of the graphs of two functions $\phi_{1}$, $\phi_{2}$ separated
    by $\tilde{\rho}$.  If the sign changes from minus to plus, then the
    two functions are defined only for $x\geq\bar{x}$, i.\,e.\ the initial
    value problem has two solutions starting in $\bar{x}$.  In the opposite
    case the functions are defined only for $x\leq\bar{x}$, i.\,e.\ the
    initial value problem has two solutions ending in $\bar{x}$.
  \end{enumerate}
\end{theorem}

\begin{remark}
  In the first case of Theorem \ref{thm:exnuni}, the corresponding initial
  value problem is obviously unsolvable, as not even a weak geometric
  solution exists.  In the first-order case, the projection of the weak
  generalised solution is a vertical line which is of some interest, if it
  extends over the whole $u$-axis, as it then represents a separatrix.  In
  the absence of proper impasse points on this line, no other weak
  geometric solution can cross it.  Hence, if $\bar{x}$ is the
  $x$-coordinate of the line, then no strong solution can be defined on
  an open interval $(\alpha,\omega)$ that contains $\bar{x}$.

  In all other cases, the initial value problem at an improper impasse
  point is solvable only in a weak sense, as the weak geometric solution is
  the graph of a function which is at the initial point only $q-1$ times
  differentiable.  In the second case we have then a unique weak solution,
  whereas in the third and generic case the initial value problem has
  exactly two weak solutions (and depending on the direction of the sign
  change one should speak of a ``terminal value problem'').  A classical
  analytical proof of a similar result can be found in
  \cite[Thm.~4.1]{pjr:singular}.
\end{remark}

We can now continue the discussion started after Proposition
\ref{prop:impasse} of situations that cannot be handled by a classical
fully nonlinear analysis.

\begin{example}\label{ex:foeq}
  A scalar first-order quasi-linear differential equation is of the form
  $g(x,u)u'=f(x,u)$ and thus defined by two functions
  $f,g:\mathcal{J}_{0}\rightarrow\RR$.  Our approach requires then the
  analysis of the vector field $Y=g(x,u)\partial_{x}+f(x,u)\partial_{u}$
  defined everywhere on the plane $\mathcal{J}_{0}$.  Thus obviously any
  phenomenon that can appear at a stationary point of a planar vector field
  may also arise at a proper impasse point of first-order quasi-linear
  equation.

  Let $\tilde{\rho}=(\bar{x},\bar{u})$ be a proper impasse point of our
  equation, i.\,e.\ a stationary point of the field $Y$.  The local
  behaviour of the trajectories of $Y$ around $\tilde{\rho}$ is largely
  determined by the Jacobian of $Y$ evaluated at $\tilde{\rho}$; we call it
  $\tilde{J}$.  The genericity condition of Proposition \ref{prop:impasse}
  takes the simple form $g_{u}(\tilde{\rho})=0$ and the discriminant
  deciding on the existence of real irregular singularities in the fibre
  over $\tilde{\rho}$ is given by
  \begin{equation}\label{eq:delta}
    \Delta=\bigl[f_{u}(\tilde{\rho})-g_{x}(\tilde{\rho})\bigr]^{2} +
    4f_{x}(\tilde{\rho})g_{u}(\tilde{\rho})\,.
  \end{equation}
  In a straightforward computation one can show that $\Delta$ is also the
  discriminant of the characteristic polynomial of $\tilde{J}$.

  Thus if there are only complex irregular singularities above
  $\tilde{\rho}$, then the eigenvalues of $\tilde{J}$ are complex, too.  If
  their real part does not vanish, then infinitely many weak geometric
  solutions approach the impasse point.  However, none of them has a well
  defined tangent in the limit.  This implies that it is not possible to
  combine two of them with the impasse point to a $\mathcal{C}^{1}$
  manifold.  Hence no strong solution can go through the impasse point.  If
  the eigenvalues are purely imaginary, then no weak geometric solution
  approaches the impasse point.

  A concrete example for the former case is provided by the functions
  $g(x,u)=u$ and $f(x,u)=u-x$.  Here the impasse points form the $x$-axis
  in $\mathcal{J}_{0}=\RR^{2}$ and the only proper impasse point is the
  origin.  The eigenvalues of the Jacobian of the (here even linear) vector
  field $Y$ at the origin are $\tfrac{1}{2}(1\pm\sqrt{3}i)$.  All weak
  geometric solutions spiral towards the origin.  Only those parts of them
  that lie completely either in the upper or in the lower half plane
  represent the graph of strong solution.  The improper impasse points
  represent turning points where we see the behaviour described in Theorem
  \ref{thm:exnuni}(iii) with two strong solutions either starting or
  ending.

  As a concrete example corresponding to the case (iii) of Proposition
  \ref{prop:impasse}, we consider the choice $g(x,u)=x^{2}$ and
  $f(x,u)=u^{2}+x$.  Again the origin is the only proper impasse point and
  there do not exist any irregular singularities in the fibre above it.  At
  none of the points in this fibre the gradient $(2xu_1 - 1,-2u, x^2)$ of
  the defining equation for $\mathcal{R}_1$ vanishes.  Thus the fibre does
  not contain any algebraic singularity.  We study the phase portrait of
  the vector field
  \begin{equation}\label{eq:foql}
    Y= x^2 \partial_x + (u^2 + x) \partial_u
  \end{equation}
  near the origin which is obviously an isolated stationary point.  One
  easily verifies that one is dealing with a nilpotent stationary point
  which is, according to \cite[Theorem 3.5 (4.i2)]{dla:qualplan}, a
  saddle-node.

  \begin{figure}[ht]
    \centering
    \includegraphics[width=6cm]{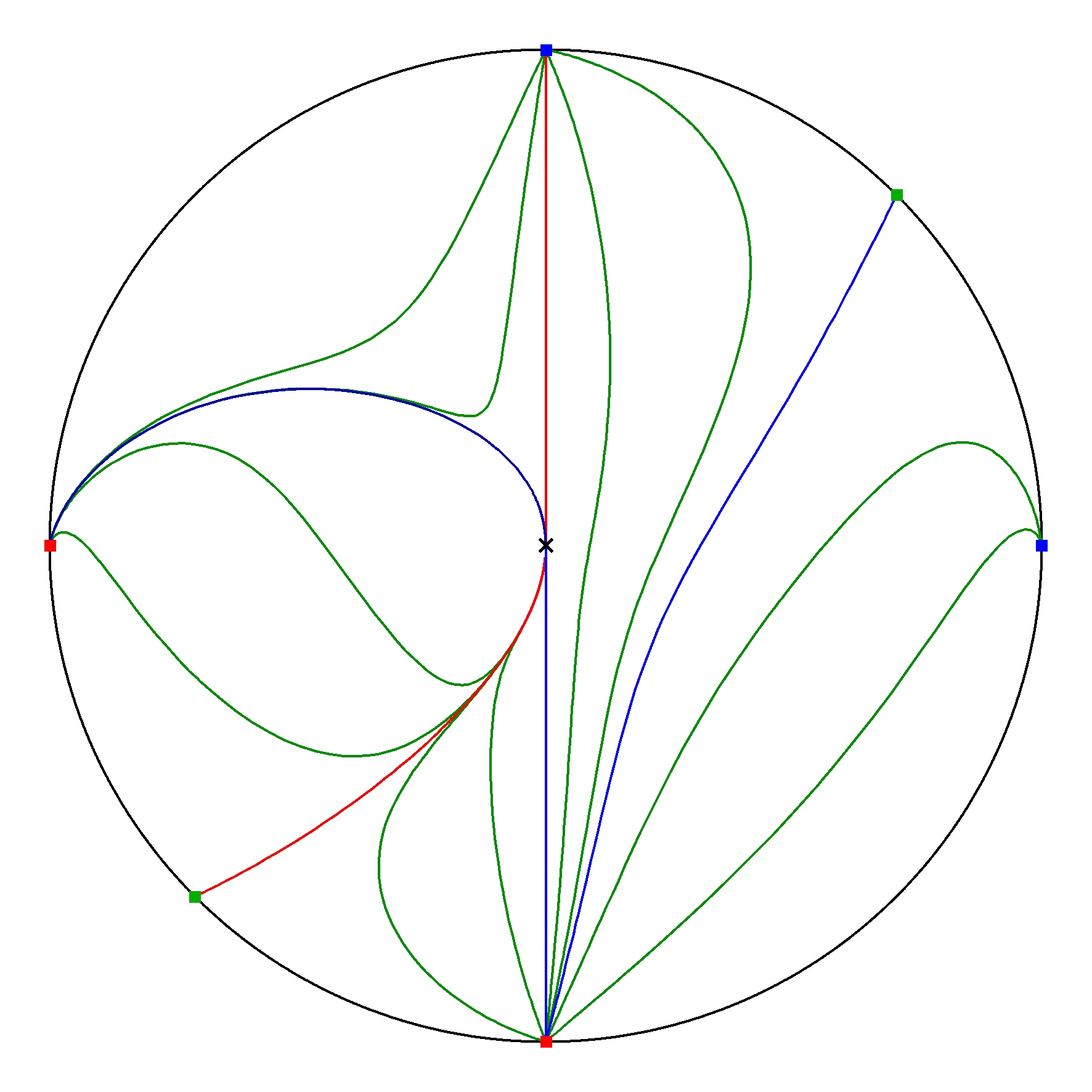}
    \caption{Phase portrait of \eqref{eq:foql} on the Poincar\'e disc}
    \label{fig:foql}
  \end{figure}

  For a detailed analysis of the phase portrait, we used the programme P4
  (described in \cite[Chapt.~9]{dla:qualplan}).  Via quasi-homogeneous
  blow-ups, it determines the existence of two hyperbolic and two parabolic
  sectors and computes Taylor approximations of the separatrices between
  the sectors.  Figure \ref{fig:foql} shows the phase portrait of
  \eqref{eq:foql} on the Poincar\'e disc with the origin at the centre (see
  \cite[Chapt.~5]{dla:qualplan} for a detailed description of this
  presentation).  The blue and red curves represent separatrices and in
  each region two representative trajectories are plotted in green.  The
  two parabolic sectors are in the lower left part of the Poincar\'e disc
  bounded by two blue separatrices and separated by a red separatrix.

  For the original quasi-linear equation, we can make the following
  observations based on this phase portrait.  The only weak generalised
  solution going through the origin is the $u$-axis which does not induce a
  geometric solution.  Hence the initial value problem $u(0)=0$ does not
  possess any two-sided solutions.  All one-sided solutions are only
  defined for $x\leq0$.\footnote{If one takes a closer look at the surface
    $\mathcal{R}_{1}\subset\mathcal{J}_{1}$ defined by our equation, then
    one sees that it consists of two disjoint components: one containing
    all points with $x\leq0$ and one containing all points with $x>0$.
    Only asymptotically the two components meet in the ``point''
    $(0,0,\infty)$.}  The Taylor approximations of the red and the blue
  separatrices in the left half of the Poincar\'e disc show that they both
  enter the origin with a vertical tangent.  As all trajectories inside the
  two parabolic sectors tend asymptotically towards the red separatrix, all
  one-sided geometric solutions have a vertical tangent at the origin,
  implying that the corresponding functions are not differentiable for
  $x=0$.  Hence, our initial value problem possesses a one-parameter family
  of solutions each living in
  $\mathcal{C}^{0}\bigl((-\delta,0]\bigr)\cap
  \mathcal{C}^{\infty}\bigl((-\delta,0)\bigr)$ for some $\delta>0$.  All
  solutions in the upper parabolic sector are actually defined on the whole
  negative real axis and tend for $x\rightarrow-\infty$ against a finite
  value.  By contrast, all solutions in the lower parabolic sector become
  singular at a finite value $\delta>0$ depending on the solution.
\end{example}

\section{Proper Impasse Points}
\label{sec:propimp}

It follows immediately from their definition that proper impasse points are
stationary for the vector field $Y$.  Hence the analysis of the local
solution behaviour in their neighbourhood requires to understand the local
phase portrait of $Y$ at a stationary point.  However, one should note
important differences in the \emph{interpretation} of the results of such
an analysis.  For studying a quasi-linear equation, we do not really need
the specific vector field $Y$, but only the distribution generated by it.
Thus instead of $Y$ we may also take any vector field obtained by
multiplying $Y$ with a nowhere vanishing function.  As this includes the
field $-Y$, absolute signs of eigenvalues have no meaning for us; only
relative signs matter.  Furthermore, according to our definition of a weak
generalised solution, we are not directly interested in trajectories but in
one-dimensional invariant manifolds passing through the impasse point.
Hence for deciding the existence of weak generalised solutions of a certain
regularity, it is necessary to analyse whether it is possible to combine
two trajectories of $Y$ approaching the impasse point with the impasse
point to an invariant curve of the desired regularity.

In the case that the proper impasse point is a hyperbolic stationary point,
the stable manifold theorem asserts the existence of a unique stable and a
unique unstable manifold which are both smooth under our assumptions.  Any
trajectory approaching the impasse point must lie on one of these two
manifolds.  Hence also weak generalised solutions can only exist on these
manifolds.  If the (un)stable manifold is one-dimensional, it is
simultaneously a weak generalised solution.  However, it is still possible
that the manifold is vertical and thus does not define a proper weak
generalised solution.

Unfortunately, it is easy to see that for any scalar quasi-linear equation
\eqref{eq:qlode} of an order $q>1$, no proper impasse point can be
hyperbolic.  Indeed the special form of the projected vector field $Y$
given by \eqref{eq:projvd} trivially implies that its Jacobian possesses at
any point where $g$ vanishes zero eigenvalues.  This implies that the
situation becomes much more complicated, as the analysis of the centre
manifolds is more delicate.  First of all, we find in general infinitely
many centre manifolds.  For applications like centre manifold reductions,
these are usually considered as equivalent, as asymptotically they are
exponentially close.  If we assume for a moment that the centre manifolds
are one-dimensional, then each represents in our point of view a different
weak generalised solution and thus matters for uniqueness questions.
Secondly, the regularity of the centre manifolds may drop compared to the
regularity of the functions $f$ and $g$.  Finally, if a one-dimensional
centre manifold consists entirely of stationary points (and then is
automatically unique by \cite[Cor.~3.3]{js:cm}), it generally does not
define a proper weak generalised solution and hence does not even lead to a
weak geometric solution.  Higher-dimensional centre manifolds require a
much more sophisticated analysis (using e.\,g.\ normal forms), as many
different possibilities arise.

Finally, we come back to the remark following Proposition
\ref{prop:impasse} and consider the case that the fibre over the proper
impasse point does not contain an irregular singularity.  We will now show
that in such a ``truly quasi-linear'' situation only weak solutions may
exist.

\begin{proposition}\label{prop:qlimp}
  Let $\tilde{\rho}\in\mathcal{J}_{q-1}$ be a proper impasse point of the
  scalar quasi-linear equation $\mathcal{R}_{q}\subset\mathcal{J}_{q}$ such
  that for the restricted projection
  $\hat{\pi}_{q-1}^{q}:\mathcal{R}_{q}\rightarrow\mathcal{J}_{q-1}$ the
  fibre $\bigl(\hat{\pi}_{q-1}^{q}\bigr)^{-1}(\tilde{\rho})$ does not
  contain any irregular singularity.  Then there does not exist a strong
  solution $\phi$ with $\tilde{\rho}\in\im{j_{q-1}\sigma_{\phi}}$.  The
  fibre
  $\bigl(\hat{\pi}_{q-1}^{q}\bigr)^{-1}(\tilde{\rho})\subset\mathcal{J}_{q}$
  is a generalised solution and it does not intersect with any other
  generalised solution.
\end{proposition}

\begin{proof}
  For a scalar quasi-linear differential equation of the form
  \eqref{eq:qlode}, a point $\tilde{\rho}$ is a proper impasse point if and
  only if $g(\tilde{\rho})=f(\tilde{\rho})=0$.  But this implies
  immediately that the fibre
  $\bigl(\hat{\pi}_{q-1}^{q}\bigr)^{-1}(\tilde{\rho})$ is the whole line
  $\bigl(\pi_{q-1}^{q}\bigr)^{-1}(\tilde{\rho})$.  If it does not contain
  an irregular singularity, then at every point in it the Vessiot space is
  one-dimensional and vertical, i.\,e.\ tangential to the fibre.  Hence the
  fibre defines a generalised solution.  Since generalised solutions may
  intersect only in irregular singularities, no other one contains a point
  of the fibre.  But this observation immediately implies the non-existence
  of a strong solution $\phi$ with
  $\tilde{\rho}\in\im{j_{q-1}\sigma_{\phi}}$: if such a $\phi$ existed,
  $\im{j_{q}\sigma_{\phi}}$ would be a generalised solution intersecting
  the fibre (at the point corresponding to the value of $q$th derivative of
  $\phi$).
\end{proof}

Note that it is nevertheless possible that the initial value problem
defined by the point $\tilde{\rho}$ possesses one or more weak geometric
solutions.  However, none of these can correspond to a
$\mathcal{C}^{q}$-function.

\section{Second-Order Initial Value Problems}
\label{sec:soivp}

For a quasi-linear second-order equation $g(x,u,u')u''=f(x,u,u')$, the
vector field $Y=gC^{(1)}+fC_{1}$ generating the projected Vessiot
distribution lives on the three-dimensional manifold $\mathcal{J}_{1}$.  In
contrast to the situation in the first-order case, we cannot produce any
three-dimensional vector field
$Z=a\partial_{x}+b\partial_{u}+c\partial_{u'}$, but only those satisfying
the ``syzygy'' $u'a-b=0$.  As already mentioned, this constraint
immediately excludes the existence of hyperbolic stationary points for
$Y$. The analysis of non-hyperbolic stationary points of three-dimensional
dynamical systems may become highly non-trivial and there does not yet
exist a complete theory as in planar case.

We will consider here the special case that the function $g$ depends only
on one variable, as it leads to considerable simplifications.  As many
equations arising in concrete applications are of this particular form, it
is also of practical relevance.  We will analyse in detail the case that
$g=g(x)$ and that the initial condition is of the form $u(y)=c_{0}$ and
$u'(y)=c_{1}$ with $y$ a \emph{simple} zero of $g$.  Liang
\cite{jfl:singivp} studied this situation for $g(x)=x$ (and $y=0$) with
classical analytical techniques.  We will show how the results in
\cite{jfl:singivp} can be reproduced with our geometric techniques.  It
will turn out that our techniques are not only more straightforward and
``automatic'', as they do not require steps like guessing of a good
integrating factor or finding the right estimates, but they also provide a
much clearer explanation of the findings.  The two other univariate cases
$g=g(u)$ and $g=g(u_{1})$ behave -- somewhat surprisingly -- rather similar
and will be studied elsewhere.

To obtain regularity results, we need to study not only the original
differential equation $\mathcal{R}_{2}\subset\mathcal{J}_{2}$ defined as
the zero set of $F_{2}(x,\uv_{(2)})=g(x)u''-f(x,\uv_{(1)})$, but also all
its prolongations $\mathcal{R}_{q}\subset\mathcal{J}_{q}$ for $q>2$.  They
are given by the zero sets of the functions
\begin{equation}\label{eq:Fq}
  F_{q}(x,\uv_{(q)})=g(x)u^{(q)} + 
                  \bigl[(q-2)g'(x)-f_{u'}(x,\uv_{(1)})\bigr]u^{(q-1)} -
                  h_{q}(x,\uv_{(q-2)})
\end{equation}
where the contributions of the lower-order derivatives are collected in
functions $h_{q}$ recursively defined for $q>2$ by
\begin{equation}\label{eq:hq}
  \begin{aligned}
    h_{3}(x,\uv_{(1)}) &= C^{(1)}f(x,\uv_{(1)})\,,\\
    h_{q}(x,\uv_{(q-2)}) &= C^{(q-2)}\Bigl(h_{q-1}(x,\uv_{(q-3)})-
        \bigl[(q-3)g'(x)-f_{u'}(x,\uv_{(1)})\bigr]u^{(q-2)}\Bigr)\,.
  \end{aligned}
\end{equation}
Obviously, $\dim{\mathcal{R}_{q}}=3$ for all $q\geq1$.

We determine first the singular points on the differential equations
$\mathcal{R}_{q}$ for any $q\geq2$.  If $\rho_{q}=(\bar{x},\bar{\uv}_{(q)})$
is a point on $\mathcal{R}_{q}$, we denote by
$\rho_{k}=\pi^{q}_{k}(\rho_{q})$ its projection to $\mathcal{J}_{k}$ for
any $0\leq k<q$.  We note that $\mathcal{R}_{2}$ is a manifold except at
points $\rho_{2}$ with $g(\bar{x})=0$, $f_{u}(\rho_{1})=f_{u'}(\rho_{1})=0$
and $f_{x}(\rho_{1})=g'(\bar{x})\bar{u}''$.  Together with the differential
equation itself, this represents five conditions for four coordinates.
Thus generically $\mathcal{R}_{2}$ is everywhere a manifold.  By a similar
argument, the same is true also for all prolongations $\mathcal{R}_{q}$.
We will therefore assume from now on that no algebraic singularities appear
at any order.

\begin{lemma}\label{lem:singL}
  For any order $q\geq2$, the point
  $\rho_{q}=(\bar{x},\bar{\uv}_{(q)})\in\mathcal{R}_{q}$ is singular, if and
  only if $g(\bar{x})=0$.  It is an irregular singular point, if and only
  if in addition
  \begin{displaymath}
    \bigl[(q-1)g'(\bar{x})-f_{u'}(\rho_{1})\bigr]\bar{u}^{(q)}=
    h_{q+1}(\rho_{q-1})\,.
  \end{displaymath}
\end{lemma}

\begin{proof}
  If we make the ansatz $X^{(2)}=aC^{(2)}|_{\rho_{2}}+bC_{2}|_{\rho_{2}}$
  for a vector in the Vessiot space
  $\mathcal{V}_{\rho_{2}}[\mathcal{R}_{2}]$, then we obtain for the
  coefficients $a,b$ the following linear system:
  \begin{displaymath}
    \Bigl(\bigl[g'(\bar{x})-f_{u'}(\rho_{1})\bigr]\bar{u}''-
           h_{3}(\rho_{1})\Bigr)a + g(\bar{x})b=0\,.
  \end{displaymath}
  For an irregular singularity, both coefficients must vanish.  For a
  regular singularity the coefficient of $b$ must vanish, whereas the
  coefficient of $a$ must be non-zero.  This proves the assertion for
  $q=2$.

  For $q>2$, we make the corresponding ansatz
  $X^{(q)}=aC^{(q)}|_{\rho_{q}}+bC_{q}|_{\rho_{q}}$ and obtain the linear
  system (cf.~Remark \ref{rem:vdprol}):
  \begin{displaymath}
    \Bigl(\bigl[(q-1)g'(\bar{x})-f_{u'}(\rho_{1})\bigr]\bar{u}^{(q)}-
           h_{q+1}(\rho_{q-1})\Bigr)a + g(\bar{x})b=0\,.
  \end{displaymath}
  The assertion follows now by the same argument as above.
\end{proof}

We will assume in the sequel that $\rho_{1}$ is determined by the initial
data of our initial value problem:
$\rho_{1}=(\bar{x}=y,\bar{u}_{0}=c_{0},\bar{u}_{1}=c_{1})$ with $y$ a
simple zero of $g$ so that we are indeed at an impasse point.  Furthermore,
we will set $\delta=g'(y)$ and $\gamma=f_{u'}(\rho_{1})$.  Note that the
assumption of a simple zero implies that $\delta\neq0$.
Lemma~\ref{lem:singL} indicates that a special case arises when $\gamma$ is
an integral multiple of $\delta$.

\begin{definition}\label{def:resL}
  The singular initial value problem determined by
  $\rho_{1}\in\mathcal{J}_{1}$ has a \emph{resonance at order $k\in\NN$},
  if $k\delta=\gamma$.  In this case, we introduce at any point
  $\rho_{k}\in(\pi^{k}_{1})^{-1}(\rho_{1})$ above $\rho_{1}$ the
  \emph{resonance parameter} $A_{k}=h_{k+2}(\rho_{k})$ and call the
  resonance \emph{critical} at $\rho_{k}$ for $A_{k}\neq0$ and
  \emph{smooth} at $\rho_{k}$ for $A_{k}=0$.
\end{definition}

\begin{remark}
  It will later turn out that only a unique point
  $\rho_{k}\in(\pi^{k}_{1})^{-1}(\rho_{1})$ is of relevance for the initial
  value problem considered by us.  Depending on the value of the resonance
  parameter at this point, we will simply speak of a critical or smooth
  resonance, respectively.  It is not difficult to see that the
  calculations forming the core of the proof of \cite[Lemma
  3.1]{jfl:singivp} are equivalent to those underlying the proof of
  Lemma~\ref{lem:singL} for the special case $g(x)=x$.  Hence, if in this
  special case a resonance occurs at order $k$, then Liang's parameter $A$
  is exactly our resonance parameter $A_{k}$.
\end{remark}

\begin{corollary}\label{cor:fibreL}
  Let $\rho_{q}\in\mathcal{R}_{q}$ be an irregular singularity.  Then the
  whole fibre $\mathcal{F}_{q+1}=(\pi^{q+1}_{q})^{-1}(\rho_{q})$ is
  contained in the prolonged equation $\mathcal{R}_{q+1}$.  If the initial
  value problem is not in resonance at order $q$, then $\mathcal{F}_{q+1}$
  contains exactly one irregular singularity.  In the case of a critical
  resonance at order $q$, the fibre $\mathcal{F}_{q+1}$ consists entirely
  of regular singularities.  If the resonance is smooth, then all points in
  $\mathcal{F}_{q+1}$ are irregular singularities.
\end{corollary}

\begin{proof}
  The first assertion was already shown in Proposition \ref{prop:fibre}.
  According to Lemma \ref{lem:singL}, a point
  $\rho_{q+1}\in\mathcal{F}_{q+1}$ is an irregular singularity, if and only
  if its highest component $\bar{u}^{(q+1)}$ satisfies the equation
  $(q\delta-\gamma)\bar{u}^{(q+1)}=A_{q}$.  If the initial value problem
  has not a resonance at order $q$, then this condition determines
  $\bar{u}^{(q+1)}$ uniquely.  In the case of a smooth (critical)
  resonance, this condition is satisfied by any (no) value
  $\bar{u}^{(q+1)}$.
\end{proof}

From the proof of Lemma \ref{lem:singL}, it is straightforward to obtain a
generator $X^{(q)}$ for the Vessiot distribution
$\mathcal{V}[\mathcal{R}_{q}]$ outside of the irregular singularities of
$\mathcal{R}_{q}$.  Since $\mathcal{R}_{q}$ is quasi-linear, we are more
interested in the projected Vessiot distribution
$(\pi^{q}_{q-1})_{*}\mathcal{V}[\mathcal{R}_{q}]$.  For $q=2$, it is
generated by the vector field
\begin{displaymath}
  Y^{(1)}=g(x)\partial_{x}+g(x)u'\partial_{u}+f(x,\uv_{(1)})\partial_{u'}
\end{displaymath}
and for an arbitrary order $q>2$ by the field
\begin{equation}\label{eq:Y1}
  \begin{aligned}
    Y^{(q-1)} &{}= g(x)C^{(q-1)}\\
                  &{}+ \Bigl(h_{q}(x,\uv_{(q-2)})- \bigl[(q-2)g'(x)-
                        f_{u'}(x,\uv_{(1)})\bigr]u^{(q-1)}\Bigr)C_{q-1}\,.
  \end{aligned}
\end{equation}
Obviously, each of these vector fields can be extended to the whole
corresponding equation $\mathcal{R}_{q-1}$.  In the sequel, we will indeed
consider them as vector fields on these three-dimensional manifolds and
study the corresponding autonomous dynamical systems.  As we do not have an
explicit parametrisation of $\mathcal{R}_{q-1}$, we have expressed
$Y^{(q-1)}$ in the full set of coordinates of $\mathcal{J}_{q-1}$ which
means that we have extended $Y^{(q-1)}$ to the whole jet bundle
$\mathcal{J}_{q-1}$.  However, this extension is not uniquely defined.
Using the equations defining $\mathcal{R}_{q-1}$, we may instead consider
for example the vector field
\begin{equation}\label{eq:Y2}
  \begin{aligned}
    \hat{Y}^{(q-1)} ={}& g(x)\partial_{x} + g(x)u'\partial_{u} +
                      f(x,\uv_{(1)})\partial_{u'} + {}\\
      & \sum_{k=3}^{q}\Bigl(h_{k}(t,\uv_{(k-2)})-
            \bigl[(k-2)g'(x)-f_{u'}(x,\uv_{(1)})\bigr]u^{(k-1)}\Bigr)
            \partial_{u^{(k-1)}}
  \end{aligned}
\end{equation}
which coincides with $Y^{(q-1)}$ on $\mathcal{R}_{q-1}$ but not on the
rest of $\mathcal{J}_{q-1}$.

Recall from above that we assume that the projection
$\rho_{1}=(y,c_{0},c_{1})$ defines our singular initial data, i.\,e.\
$g(y)=0$.  Thus all points
$\rho_{q}=(\bar{x},\bar{\uv}_{q})\in\mathcal{R}_{q}\cap(\pi^{q}_{1})^{-1}(\rho_{1})$
with $\bar{x}=y$, $\bar{u}=c_{0}$ and $\bar{u}'=c_{1}$ are singular, too.
If we choose one of them, then a comparison of \eqref{eq:Fq} and
\eqref{eq:Y1} shows immediately that the projection
$\rho_{q-1}=\pi^{q}_{q-1}(\rho_{q})\in\mathcal{R}_{q-1}$ is a proper
impasse point of $\mathcal{R}_{q}$ and hence a stationary point of both
$Y^{(q-1)}$ and $\hat{Y}^{(q-1)}$.  For an analysis of the local phase
portrait, we need the Jacobian at $\rho_{q-1}$.  A straightforward
computation yields for the field $Y^{(q-1)}$
\begin{equation}\label{eq:JacY1}
  J^{(q-1)}=
  \begin{pmatrix}
    \delta & 0 & \cdots & 0\\
    \delta \bar{u}' & 0 & \cdots & 0\\
    \vdots & \vdots && \vdots\\
    \delta \bar{u}^{(q-1)} & 0 & \cdots & 0\\
    a_{0} & \cdots & a_{q-1} & \gamma-(q-2)\delta
  \end{pmatrix}
\end{equation}
where the parameter $a_{0}$, \dots, $a_{q}$ are placeholders for
complicated expressions in $y$, $\bar{\uv}_{q-1}$.  Obviously, its
eigenvalues are $\delta$, $\gamma-(q-2)\delta$ and $(q-1)$ times $0$.
Recall that $Y^{(q-1)}$ should be considered as a vector field on the
three-dimensional manifold $\mathcal{R}_{q-1}$ and therefore the question
arises which three of these eigenvalues are the relevant ones?  Brute force
approaches would consist of either computing the corresponding
(generalised) eigenvectors and checking which ones are tangent to
$\mathcal{R}_{q-1}$ or of constructing an explicit parametrisation of
$\mathcal{R}_{q-1}$.

However, it turns out that a simpler possibility exists: we do the same
computations for the field $\hat{Y}^{(q-1)}$ where we get for the Jacobian
at $\rho_{q-1}$
\begin{equation}\label{eq:JacY2}
  \hat{J}^{(q-1)}=
  \begin{pmatrix}
    \delta & 0 &&&&& 0\\
    \delta c_{1} & 0 &&&&&\\
    f_{x}(\rho_{1}) & f_{u}(\rho_{1}) & \gamma &&&&\\
    &&& \gamma-\delta &&&\\
    &&&& \gamma-2\delta &&\\
    &&&&& \ddots &\\
    \star &&&&&& \gamma-(q-2)\delta
  \end{pmatrix}.
\end{equation}
Again, it is straightforward to determine the eigenvalues which are all
distinct if we do not have a resonance at an order less than $q-1$.  

We will see later that there is no need to consider $Y^{(q-1)}$, if there
is a resonance at an order less than $q-1$. Hence we exclude these cases.
Comparing now the results for the two vector fields, we conclude that the
relevant eigenvalues are $\delta$, $0$ and $\gamma-(q-2)\delta$.  If there
is a resonance at order $q-1$, then the first and the last one are equal;
otherwise all eigenvalues are distinct.

Let us first consider the case without resonance.  The eigenspace for the
eigenvalue $\gamma-(q-2)\delta$ is obviously spanned by
$(0,\dots,0,1)^{T}$.  The eigenspace for the eigenvalue $0$ is also easy to
interpret using the vector field $\hat{Y}^{(q-1)}$.  The set of all proper
impasse points of $\mathcal{R}_{q}$ is a curve on $\mathcal{R}_{q-1}$
described by the singularity condition $g(x)=0$ and the $q-1$ equations of
$\mathcal{R}_{q}$ (which can be interpreted as equations on
$\mathcal{J}_{q-1}$ when $g(x)=0$, as then nowhere $u^{(q)}$ appears).  It
follows from Remark~\ref{rem:vdprol} that the kernel of $\hat{J}^{(q-1)}$
describes the tangent space of this curve at the point $\rho_{q-1}$.
Finally, explicitly writing out the expressions for the placeholders
$a_{i}$ and comparing with the recursive definition of $h_{q+1}$, we find
that an eigenvector for the eigenvalue $\delta$ is
\begin{displaymath}
  \Bigl(1,\bar{u}',\dots,\bar{u}^{(q-1)},
  \frac{-h_{q+1}(y,\bar{u},\dots,\bar{u}^{(q-1)})}
  {\gamma-(q-1)\delta}\Bigr)^{T}\,.
\end{displaymath}
Note that it is the only eigenvector transversal to the projection
$\pi^{q-1}$.

In the case of a resonance at order $q-1$, it is easy to see that one
eigenvector for the eigenvalue $\delta$ is given by $(0,\dots,0,1)^{T}$.
Computing the kernel of the matrix
$(J^{(q-1)}-\delta\mathbbm{1}_{q+1})^{2}$, one obtains as second, linearly
independent (generalised) eigenvector
$(1,\bar{u}',\dots,\bar{u}^{(q-1)},0)^{T}$.  More precisely, we must
distinguish two cases. If $a_{0}+\sum_{i=1}^{q-1}a_{i}\bar{u}^{(i)}=0$,
then both vectors are proper eigenvectors.  Otherwise, the second vector is
only a generalised eigenvector and for obtaining the basis leading to the
Jordan normal form one must divide it by
$a_{0}+\sum_{i=1}^{q-1}a_{i}\bar{u}^{(i)}$.  If one inserts the explicit
expressions for the placeholders $a_{i}$, then it is not difficult to see
from our formula for the prolonged equations that the resonance parameter
$A_{q-1}$ is given by the sum
$a_{0}+\sum_{i=1}^{q-1}a_{i}\bar{u}^{(i)}=
h_{q+1}(y,\bar{u},\dots,\bar{u}^{(q-1)})$.  Hence the above case
distinction corresponds to the question whether or not the resonance is
smooth.

Based on these observations, we can now give a complete overview over the
existence, (non)uniqueness and regularity of solutions for the studied
singular initial value problem.  It recovers in our slightly more general
situation all the results of Liang \cite{jfl:singivp} except that we
describe the asymptotic behaviour of the solutions as they approach the
singularity in a different way.  We begin with the case that no resonance
appears.

\begin{theorem}\label{thm:Lnores}
  Consider for the differential equation $g(x)u''=f(x,u,u')$ the initial
  value problem determined by the point $\rho_{1}=(y,c_{0},c_{1})$ with $y$
  a simple zero of $g$ and $f(y,c_{0},c_{1})=0$.  We set $\delta=g'(y)$,
  $\gamma=f_{u'}(\rho_{1})$ and assume that at no order a resonance
  appears.
  \begin{enumerate}[label=(\roman*)]
  \item If $\delta\gamma<0$, then the initial value problem possesses a
    unique two-sided smooth solution and no additional one-sided solutions.
  \item If $\delta\gamma>0$, then there exists a one-parameter family of
    two-sided solutions.  One of these solutions is smooth; the other ones
    are in $\mathcal{C}^{k}\setminus\mathcal{C}^{k+1}$ with their
    regularity given by $k=\lceil\gamma/\delta\rceil$.  All of these
    solutions possess the same Taylor polynomial
    $\sum_{i=0}^{k}\tfrac{c_{i}}{i!}(x-y)^{i}$ of degree $k$ around $y$ and
    each of them is uniquely characterised by the limit
    \begin{displaymath}
      \lim_{x\rightarrow y}\frac{u^{(k)}(x)-c_{k}}
      {|x-y|^{(\gamma-(k-1)\delta)/\delta}}\,.
    \end{displaymath}
  \end{enumerate}
\end{theorem}

\begin{proof}
  Consider the fibre $\mathcal{F}_{2}=(\pi^{2}_{1})^{-1}(\rho_{1})$.  It is
  trivial to see that $\mathcal{F}_{2}\subset\mathcal{R}_{2}$.  Because of
  the absence of resonances, it follows from Lemma~\ref{lem:singL} and
  Corollary~\ref{cor:fibreL} that $\mathcal{F}_{2}$ contains exactly one
  irregular singularity $\rho_{2}=(y,c_{0},c_{1},c_{2})$; all other points
  in $\mathcal{F}_{2}$ with $\bar{u}^{(2)}\neq c_{2}$ are regular
  singularities.  As discussed in the proof of Theorem~\ref{thm:regsing},
  the unique generalised solution through any of these latter points is the
  fibre $\mathcal{F}_{2}$ itself.  As this is obviously not a proper
  generalised solution, we conclude that the second prolongation of any
  $\mathcal{C}^{2}$ solution of our initial value problem must pass through
  $\rho_{2}$.

  We can now proceed by induction.  If $\rho_{q}$ is the unique irregular
  singularity in the fibre $\mathcal{F}_{q}$, then Proposition
  \ref{prop:fibre} entails that the entire fibre $\mathcal{F}_{q+1}$ over
  $\rho_{q}$ is contained in $\mathcal{R}_{q+1}$.  By Lemma~\ref{lem:singL}
  and Corollary~\ref{cor:fibreL}, it contains a unique irregular
  singularity $\rho_{q+1}$.  By the same argument as above, the
  prolongation of order $q+1$ of any $\mathcal{C}^{q+1}$ solution of our
  initial value problem must pass through $\rho_{q+1}$. We denote the value
  $\bar{u}^{(q)}$ of the $u^{(q)}$-coordinate of $\rho_{q}$ by $c_{q}$,
  i.\,e.\ from now on we always assume that
  $\rho_{q}=(y,c_{0},\dots,c_{q})$.

  We consider now first the case that $\delta\gamma<0$.  Without loss of
  generality we assume that $\delta>0$, as otherwise we simply multiply our
  equation by $-1$.  The Jacobian $J^{(1)}$ of the vector field $Y^{(1)}$
  at the initial point $\rho_{1}$ has the three eigenvalues $\delta$, $0$,
  $\gamma$ all of which have a different sign under our assumption
  $\delta\gamma<0$.  It follows from the classical Centre Manifold Theorem
  (see e.\,g.\ \cite[Thm.~3.2.1]{gh:osc}) that there are three unique
  one-dimensional invariant manifolds tangent to the corresponding
  eigenvectors.  The uniqueness of the centre manifold follows here again
  from the fact that we have a whole curve of stationary points.

  Based on the above discussion of the eigenvectors, it is easy to identify
  two of the three invariant manifolds.  The stable manifold belonging to
  the negative eigenvalue $\gamma$ is simply the fibre
  $(\pi^{1}_{0})^{-1}(y,c_{0})$.  Indeed, the fibre is an invariant
  manifold, as $Y^{(1)}$ is vertical everywhere on this fibre and thus
  tangential to the fibre.  Since at $\rho_{1}$ it is tangential to the
  eigenspace $\gamma$, the claim follows from the uniqueness of the stable
  manifold.  Thus the stable manifold is not a proper weak generalised
  solution.  The centre manifold is the curve of all proper impasse points
  which is completely contained in the fibre $(\pi^{1})^{-1}(y)$ and hence
  also not a proper weak generalised solution.  Only the unstable manifold
  corresponding to the positive eigenvalue $\delta$ defines locally a
  proper weak generalised solution projecting on a weak geometric solution
  which is the graph of a classical solution.

  The smoothness of this solution can be proven by considering the
  prolongations.  At any prolongation order $q\geq2$, we find the same
  picture.  The Jacobian $J^{(q)}$ of the vector field $Y^{(q)}$ at the
  point $\rho_{q}$ has the three eigenvalues $\delta>0$, $0$,
  $\gamma-(q-1)\delta<0$.  Only the unstable manifold defines a proper
  generalised solution projecting on a geometric solution which is the
  graph of a function.  Because of the uniqueness of the unstable manifold,
  we obtain always the same geometric solution which is thus smooth.  It
  follows immediately from our construction that this solution is two-sided
  and that no further one-sided solutions can exist.

  In the case $\delta\gamma>0$, we find different phase portraits.
  $J^{(1)}$ has now two distinct positive and one zero eigenvalue.  Again
  the argument given above implies that the centre manifold does not define
  a proper weak generalised solution whereas the trajectories on the
  two-dimensional unstable manifold yield a one-parameter family of
  one-sided proper weak generalised solutions.  This picture remains
  qualitatively unchanged at the orders $q=2,\dots,k$ so that the
  corresponding one-sided geometric solutions are the graphs of
  $\mathcal{C}^{k}$ functions.  The Taylor polynomial of degree $k$ around
  $y$ of any of these functions is given by the $k$-jet
  $\rho_{k}=(y,c_{0},\dots,c_{k})\in\mathcal{J}_{k}$.

  More precisely, for $1\leq q\leq k$, the dynamics on the unstable
  manifold corresponds to that around an unstable two-tangent node.  If
  $q<k$, then the eigenvalue $\gamma-(q-1)\delta$ is the larger one by the
  definition of $k$. Hence almost all trajectories of $Y^{(q)}$ reach the
  point $\rho_{q}$ tangential to the transversal eigenvector belonging to
  the smaller eigenvalue $\delta$.  Thus we can always combine two
  trajectories coming from the left and the right, respectively, to a
  two-sided proper generalised solution which is the prolonged graph of a
  function of class $\mathcal{C}^{q}$.  However, at the order $q=k$ there
  is a change: now $\delta$ is the larger eigenvalue and hence almost all
  generalised solutions are tangential to the vertical eigenvector
  belonging to the eigenvalue $\gamma-(k-1)\delta$.  The verticality
  implies that the $(k+1)$th derivative of the corresponding classical
  solution becomes infinite at $x=y$.  Thus all these generalised solutions
  come from functions which are of class $\mathcal{C}^{k}$, but not of
  class $\mathcal{C}^{k+1}$ at $y$.  There is only one generalised solution
  which is tangential to the transversal eigenvector belonging to $\delta$
  and which thus corresponds to an at least $k+1$ times differentiable
  function.

  Obviously, in the above argument we are implicitly applying the
  Hartman-Grobman Theorem asserting an equivalence between the phase
  portraits of a dynamical system in the neighbourhood of a hyperbolic
  stationary point and of its linearisation around this point.  The
  standard formulation of this theorem, as one can find it in most
  textbooks like \cite{lp:deds}, asserts only a topological equivalence
  entailing that no statements about tangents are possible.  However, in
  the case of a smooth dynamical system the linearising homeomorphism is
  differentiable at the stationary point and thus preserves tangents
  \cite{ghr:dhgl}.  Hence the above statements about the tangents of the
  trajectory at the stationary point can indeed be gleaned from the
  Jacobian.

  At all orders $q>k$, the third eigenvalue is negative so that now we
  obtain qualitatively the same phase portraits as in the case
  $\delta\gamma<0$ with a one-dimensional unstable manifold.  Hence we find
  that the one solution which is at least $k+1$ times differentiable is
  actually smooth (and again trivially two-sided).  These phase portraits
  imply again that all the other members of the one-parameter family of
  $\mathcal{C}^{k}$ solutions are not contained in $\mathcal{C}^{k+1}$.

  For the remaining claim, we look at the phase portrait of $Y^{(k)}$
  around $\rho_{k}$.  As already mentioned above, $\rho_{k}$ is a
  two-tangent node for the reduced dynamics on the two-dimensional unstable
  manifold with the two eigenvalues $0<\gamma-(k-1)\delta<\delta$.  The
  trajectories of the linearised system lie in the plane spanned by the
  above computed eigenvectors and (except the two irrelevant vertical ones)
  can be written in parametrised form as
  \begin{equation}\label{eq:lintraj1}
    \begin{gathered}
    x(t)=y + \alpha e^{\delta t}\,,\quad
    u(t)=c_{0} + c_{1} \alpha e^{\delta t}\,, \dots\quad
    u^{(k-1)}(t)=c_{k-1} + c_{k} \alpha e^{\delta t}\,,\\
    u^{(k)}(t)=c_{k} -
          \frac{h_{k+2}(\rho_{k})}{\gamma-k\delta}\alpha e^{\delta t}
         + \beta e^{(\gamma-(k-1)\delta)t}\,,
    \end{gathered}
  \end{equation}
  where the constants $\alpha\neq0$ and $\beta$ may be considered as the
  coordinates of an initial point on the plane.  It suffices to consider
  $\alpha=\pm1$ to obtain all trajectories uniquely and the sign decides
  whether the trajectory is reaching $\rho_{k}$ from the left or from the
  right.  In the limit $x\rightarrow y$, the trajectories of the reduced
  system on the unstable manifold approach the ones of its linearisation.
  In \eqref{eq:lintraj1}, this limit corresponds to $t\rightarrow-\infty$.
  If we use the first equation in \eqref{eq:lintraj1} to eliminate $t$,
  then the last equation in \eqref{eq:lintraj1} shows that the
  corresponding value of the parameter $\beta$ describing the trajectory
  uniquely is obtained as the limit of
  $\bigl(u^{(k)}(x)-c_{k}\bigr)/|x-y|^{(\gamma-(k-1)\delta)/\delta}$ for
  $x\rightarrow y$, since the term proportional to $\alpha e^{\delta t}$
  vanishes in the limit.  We see furthermore from the linearised dynamics
  that we may combine the trajectories for the initial points
  $(\alpha,\beta)$ and $(-\alpha,-\beta)$, respectively, to a
  $\mathcal{C}^{1}$ curve through the stationary point $\rho_{k}$.  Hence
  the same is possible for the nonlinear reduced dynamics and by
  construction the obtained generalised two-sided solution corresponds to a
  strong $\mathcal{C}^{k}$ solution.
\end{proof}

\begin{remark}\label{rem:hoelder}
  If the functions $f$ and $g$ are even \emph{analytic}, then everywhere in
  the above theorem we can replace smooth by analytic.  Indeed, it is
  well-known that then the unstable manifold of a stationary point is also
  analytic.  The construction in our proof shows that the unstable
  manifolds are always the graph of some prolongation of our smooth
  solution.  Hence, this solution must be analytic.

  In the case of solutions with a finite regularity, one can further
  strengthen the statement of Theorem \ref{thm:Lnores}: the $k$th
  derivatives of the $\mathcal{C}^{k}$ solutions are \emph{H\"older
    continuous} with H\"older exponent
  $\lambda=(\gamma-(k-1)\delta)/\delta<1$.  This follows easily from our
  proof of Theorem \ref{thm:Lnores}.  Indeed, \eqref{eq:lintraj1} entails
  that the solution of the linearised dynamics is H\"older continuous with
  exponent $\lambda$ at $\rho_{k}$.  As already discussed above, in a
  sufficiently small open neighbourhood of $\rho_{k}$, the homeomorphism
  mapping it to the solution of the nonlinear system is $\mathcal{C}^{1}$
  and thus in particular Lipschitz continuous at $\rho_{k}$ by
  \cite{ghr:dhgl}.  It then follows from standard results about the
  composition of H\"older continuous functions that the solution of the
  nonlinear system is also H\"older continuous with exponent $\lambda$ (see
  e.\,g.\ \cite{rf:hoelder}).
\end{remark}

Before we study the effect of a resonance, we analyse the case $\gamma=0$
(ignored by Liang \cite{jfl:singivp}). It could be considered as a
resonance at order $k=0$.  However, it must be treated in a rather
different manner.  At a resonance, the relevant Jacobian has a double
eigenvalue $\delta$ and its eigenspace contains transversal vectors.  By
contrast, in the case $\gamma=0$ it possesses a double eigenvalue $0$ and
its complete eigenspace lies vertical.

\begin{theorem}\label{thm:gamma0}
  In the situation of Theorem~\ref{thm:Lnores}, assume that $\gamma=0$.
  Then there exists a unique smooth two-sided solution (and possibly
  further one-sided solutions).
\end{theorem}

\begin{proof}
  In the case $\gamma=0$, the Jacobian $J^{(1)}$ of $Y^{(1)}$ at the
  initial point $\rho_{1}$ has $0$ as a double eigenvalue with two vertical
  (generalised) eigenvectors and $\delta$ as a simple non-zero eigenvalue
  (again assumed to be positive).  Analogously to the proof of
  Theorem~\ref{thm:Lnores}, one shows that the unstable manifold is the
  graph of a prolonged smooth two-sided solution.

  The uniqueness of this solution is now a bit more subtle.  In contrast to
  the situation in the proof of Theorem~\ref{thm:Lnores}, we have now a
  two-dimensional centre manifold which is not necessarily unique.  As both
  (generalised) eigenvectors are vertical, it is also easy to see that
  there is a unique analytic centre manifold given by the plane
  $(\pi^{1})^{-1}(y)$ which, however, cannot contain any proper generalised
  solutions.  There may exist further centre manifolds (not necessarily
  smooth) and on these there could exist trajectories through $\rho_{1}$
  not contained in $(\pi^{1})^{-1}(y)$ which could correspond to the
  prolongation of a function graph.  However, even if such a trajectory
  exists, then it must have a vertical tangent in $\rho_{1}$ and hence the
  corresponding function is not twice differentiable at $y$.  Thus there
  cannot exist any further two-sided strong solutions.
\end{proof}

\begin{remark}\label{rem:gamma0}
  The theorem above leaves open the question of the existence of further
  one-sided solutions.  We will now show with the help of a simple concrete
  example that such solutions may or may not exist.  Consider the equation
  $xu''=d(u')^{m}$ with a parameter $0\neq d\in\RR$ and an exponent
  $1<m\in\NN$. For it $\delta=1$, $\gamma=0$ and $y=0$.  Its irregular
  singularities are the points $(0,c_{0},0,c_{2})$ for arbitrary values
  $c_{0},c_{2}\in\RR$.  Projection of the Vessiot distribution yields the
  dynamical system
  \begin{equation}\label{eq:gamma0}
    \dot{x}=x\,,\qquad \dot{u}=xv\,,\qquad \dot{v}=dv^{m}
  \end{equation}
  defined on $\mathcal{J}_{1}$ where we introduced $v=u'$ for notational
  simplicity.  Its stationary points are the impasse points
  $(0,c_{0},0)$.  They all possess the same unique analytic centre
  manifold, namely the plane $(\pi^{1})^{-1}(0)$.

  The first and the third equation in \eqref{eq:gamma0} form a closed
  subsystem (which is independent of $c_{0}$) with a semihyperbolic
  stationary point at the origin.  According to
  \cite[Thm.~2.19]{dla:qualplan}, we must distinguish three different
  cases:
  \begin{description}
  \item[$\boldsymbol{m}$ odd, $\boldsymbol{d<0}$] In this case, the origin
    is a saddle point of the subsystem and the only invariant manifolds
    reaching it are the centre and the unstable manifold.  Going back to
    our differential equation, we see that only two weak generalised
    solutions exist.  It follows from the form of the eigenvectors that
    only the unstable manifold of \eqref{eq:gamma0} provides us with a
    proper weak generalised solution.  Hence in this case the initial value
    problem $u(0)=c_{0}$, $u'(0)=0$ possesses only the unique two-sided
    solution from Theorem~\ref{thm:gamma0} and no further one-sided
    solutions.
  \item[$\boldsymbol{m}$ odd, $\boldsymbol{d>0}$] Now the origin is an
    unstable node of the subsystem implying the existence of many
    additional weak generalised solutions of our differential equations.
    However, all of these possess a vertical tangent at the origin and thus
    cannot be of class $\mathcal{C}^{2}$ for $x=0$.  In fact, the
    generalised solutions show a turning point behaviour at $x=0$ and thus
    each of them corresponds to two one-sided solutions which are both only
    defined either for $x\geq0$ or for $x\leq0$.
  \item[$\boldsymbol{m}$ even] This yields a combination of the two cases
    above, as the subsystem has now a saddle node at the origin.  Depending
    on the sign of $d$, we find above the unstable manifold the same phase
    portrait as for an unstable node and below as for a saddle point or
    vice versa.  In any case, we have again infinitely many additional
    one-sided solutions.
  \end{description}
  In fact, for our simple system it is straightforward to integrate the
  system \eqref{eq:gamma0} at least partially.  We find
  \begin{displaymath}
    x(t)=ae^{t}\,,\quad v(t)=\bigl(c-(m-1)dt\bigr)^{-\frac{1}{m-1}}\,,
  \end{displaymath}
  with an integration constant $c\in\RR$.  The function $u(t)$ is then
  obtained by integrating the product $x(t)v(t)$.  The result can be
  expressed in terms of the generalised exponential integrals $E_{n}(t)$.

  In the general situation of Theorem~\ref{thm:gamma0}, the stationary
  points are given by the curve $f(y,u,v)=0$ and the common unique analytic
  centre manifold of them is again a plane, namely $(\pi^{1})^{-1}(y)$.
  The form of the reduced dynamics on it in the neighbourhood of a
  stationary point $(y,c_{0},c_{1})$ is determined by the order $m$ of the
  first non-vanishing derivative
  $\tfrac{\partial^{m}f}{\partial v^{m}}(y,c_{0},c_{1})$ and the sign of
  its value $d$ and corresponds to the different cases arising in our
  simple example.  However, the value of $m$ and the sign of $d$ may now
  change at certain stationary points and there now arise further case
  distinctions.  For example, the Jacobian at the stationary point
  $(y,c_{0},c_{1})$ is diagonalisable only if the derivative
  $\tfrac{\partial f}{\partial u}(y,c_{0},c_{1})$ vanishes as in our
  example.  We refrain here from a complete analysis of all possible cases.
\end{remark}

\begin{theorem}\label{thm:Lres}
  In the situation of Theorem~\ref{thm:Lnores}, assume that a resonance
  occurs at the order $k>0$.  There exists a one-parameter family of
  two-sided solutions all possessing the same Taylor polynomial
  $\sum_{i=0}^{k}\tfrac{c_{i}}{i!}(x-y)^{i}$ of degree $k$ around $y$.  In
  the case of a smooth resonance, all of these solutions are smooth and
  each is uniquely determined by the value of its $(k+1)$st derivative in
  $y$.  In the case of a critical resonance, all solutions live in
  $\mathcal{C}^{k}\setminus\mathcal{C}^{k+1}$ and each of them is uniquely
  characterised by the value of
  \begin{displaymath}
    \lim_{x\rightarrow y}(x-y)\exp{\Bigl(-\delta\frac{u^{(k)}(x)-c_{k}}{x-y}\Bigr)}\,.
  \end{displaymath}
\end{theorem}

\begin{proof}
  We find by the same reasoning as in the proof of
  Theorem~\ref{thm:Lnores}, a unique sequence of irregular singularities
  $\rho_{j}\in\mathcal{R}_{j}$ for $j=2,\dots,k$ above the initial point
  $\rho_{1}$. At the point $\rho_{k}$ the Jacobian $J^{(k)}$ of $Y^{(k)}$
  has a double eigenvalue $\delta$ and hence there exists a unique
  two-dimensional unstable manifold.  It follows from our above analysis of
  the (generalised) eigenvectors that $\rho_{k}$ is a star node for a
  smooth resonance and a one-tangent node for a critical resonance.

  We consider first the smooth case.  Here we can always combine two
  trajectories reaching the node to a two-sided (weak in the case $k=1$)
  generalised solution.  One of these generalised solutions is vertical and
  of no interest.  All the other ones have in $\rho_{k}$ a transversal
  tangent and hence correspond locally to a solution of at least class
  $\mathcal{C}^{k+1}$.  If we apply Lemma \ref{lem:singL} with $q=k+1$,
  then it follows immediately from the definition of a smooth resonance at
  order $k$ that all points in the fibre $(\pi^{k+1}_{k})^{-1}(\rho_{k})$
  are irregular singularities of $\mathcal{R}_{k+1}$.  To each generalised
  solution through $\rho_{k}$ corresponds exactly one of these irregular
  singularities with the value $c_{k+1}$ of its $u^{(k+1)}$-coordinate
  determined by the slope of the tangent of the generalised solution in
  $\rho_{k}$.  The Jacobian $J^{(k+1)}$ of the vector field $Y^{(k+1)}$ has
  a double eigenvalue $0$.  The analysis of the local dynamics is analogous
  to the case $\gamma=0$ treated in Theorem~\ref{thm:gamma0}.  Hence we
  conclude that all these generalised solutions correspond to smooth
  two-sided solutions.  Opposed to the discussion in Remark
  \ref{rem:gamma0}, there is now no need to study the existence of further
  one-sided solutions, as these would have already shown up in our analysis
  at order $k$.

  In the case of a critical resonance, the same combination of two
  trajectories into one generalised solution is possible (for $k=1$ one
  obtains only a weak generalised solution).  All of these (weak)
  generalised solutions possess the same vertical tangent at $\rho_{k}$ and
  thus correspond locally to graphs of prolonged functions which are of
  class $\mathcal{C}^{k}$ at $y$.  Because of the vertical tangent, none of
  these solutions can be of class $\mathcal{C}^{k+1}$ at $y$.  Proceeding
  as in the proof of Theorem \ref{thm:Lnores}, one finds for the
  trajectories of the linearised dynamics after elimination of the curve
  parameter $t$ that
  $u^{(k)}(x)-c_{k}=\tfrac{x-y}{\delta}\ln{\bigl(\tfrac{x-y}{\eta}\bigr)}$
  with a constant $\eta\neq0$.  We have set the second arising constant to
  zero, as then every trajectory is uniquely described by the value of
  $\eta$.  Taking the limit $y\rightarrow x$ and solving for $\eta$ yields
  our claim.
\end{proof}

\begin{example}
  We consider the equation $xu''=(u')^{2}+x-1/4$.  All points of the form
  $(0,u,\pm1/2)$ are impasse points.  At any of them, we find $\delta=1$
  and $\gamma=\pm1$.  Thus at the points $(0,u,-1/2)$ we are in the first
  case of Theorem \ref{thm:Lnores} asserting the existence of a unique
  smooth solution through any of them.  At any of the points $(0,u,1/2)$,
  we find a resonance at order $k=1$.  As the resonance parameter is given
  by $A_{1}=1$, it is always a critical resonance and no solution can be
  twice differentiable at $x=0$.\footnote{Note that the existence of a
    resonance and its order are completely determined by the constant term
    in the equation.  If we consider the slightly more general class of
    equations $xu''=(au')^{2}+x-b^{2}$ with two parameters $a,b>0$, then we
    obtain a resonance at order $k$ if and only if $b=ka/2$.}

  The projected Vessiot distribution defines on $\mathcal{J}_{1}$ the
  dynamical system
  \begin{equation}\label{eq:liangex}
    \dot{x}=x\,,\quad \dot{u}=xv\,,\quad \dot{v}=v^{2}+x-\frac{1}{4}
  \end{equation}
  where we again abbreviated $v=u'$.  Obviously, the first and third
  equation form a closed planar system which can be explicitly integrated
  in terms of (modified) Bessel functions of the first and second kind.
  Elimination of the auxiliary curve parameter $t$ yields
  \begin{displaymath}
    u'(x)=
    \begin{cases}
      \displaystyle 
      \frac{1}{2}-\sqrt{x}\,
          \frac{CY_{0}(2\sqrt{x})+J_{0}(2\sqrt{x})}
          {CY_{1}(2\sqrt{x})+J_{1}(2\sqrt{x})} & x\geq0\\[12pt]
       \displaystyle 
       \frac{1}{2}+\sqrt{-x}\,
          \frac{CK_{0}(2\sqrt{-x})-I_{0}(2\sqrt{-x})}
                 {CK_{1}(2\sqrt{-x})+I_{1}(2\sqrt{-x})} & x\leq0
    \end{cases}
  \end{displaymath}
  with a parameter $C\in\RR$.  Here, one can explicitly verify that none of
  these solutions are twice differentiable at the origin.  $C$ is not
  exactly the parameter appearing in the proof of Theorem \ref{thm:Lres}.
  If one evaluates the limit appearing in this theorem, calling the result
  $\eta$, then in the case that one approaches the origin from the left,
  one obtains $C=2/\bigl(\ln{(-1/\eta)}-2\gamma)$, whereas for an approach
  from the right one finds $C=\pi/(2\gamma+\ln{(\eta)})$ where now $\gamma$
  denotes the Euler constant and is not related to our $\gamma$ above.  In
  any case there is thus a bijective correspondence between the limit
  $\eta$ and the parameter $C$.

  \begin{figure}[ht]
    \centering
    \includegraphics[width=6cm]{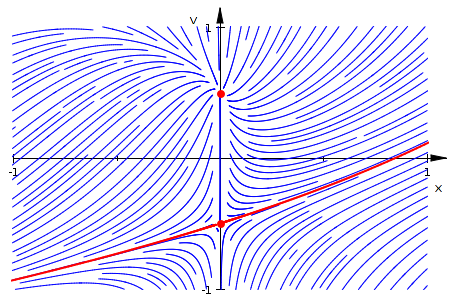}
    \caption{Phase portrait of subsystem of \eqref{eq:liangex}}
    \label{fig:liangex}
  \end{figure}

  Figure \ref{fig:liangex} shows the phase portrait of the closed subsystem
  consisting of the first and the third equation in \eqref{eq:liangex}.
  One sees that the point $(0,-1/2)$ is a saddle point and its unstable
  manifold (shown in red) is the graph of the first derivative to the
  unique solution through $(0,u,-1/2)$ (the solutions for different values
  of $u$ are all parallel to each other and thus have the same derivative).
  The stable manifold is the $v$-axis and hence irrelevant for our
  purposes.  The point $(0,1/2)$ is a one-tangent node.  The eigenvector
  again points in the $v$-direction and thus all trajectories enter the
  node with a vertical tangent.  This implies that none of the
  corresponding solutions of our second-order equation is twice
  differentiable at the origin.
\end{example}

Finally, we comment on the situation that the initial point
$\rho_{1}=(s,c_{0},c_{1})$ is chosen in such a way that $g(s)=0$ but
$f(\rho_{1})\neq0$. It is obvious that no strong, i.\,e.\ at least twice
differentiable solution can exist in this case, as no point
$\rho_{2}\in\mathcal{R}_{2}$ exists with $\pi^{2}_{1}(\rho_{2})=\rho_{1}$.
In principle, it was possible that in this case a unique two-sided solution
existed which is everywhere smooth except at $t=s$ where it is only
$\mathcal{C}^{1}$. Indeed, by our assumptions the vector field $Y^{(1)}$ is
defined everywhere on $\mathcal{J}_{1}$ and does not vanish at such a point
$\rho_{1}$. Hence, there exists a unique weak generalised solution through
$\rho_{1}$.  However, it is easy to see that it runs vertically and hence
is not proper.  Generally, this weak generalised solution will not run
through the entire fibre, as generally for some values $\tau$ we will have
$f(s,c_{0},\tau)=0$ and thus hit a point at which our above analysis
applies.  In this analysis, we mentioned that there are further weak
generalised solutions through our singularity which are, however, not
proper. One of them we have now recovered in a different manner.

\section{Conclusions}

We presented a geometric approach to the analysis of singular initial value
problems of quasi-linear ordinary differential equations.  As sketched in
the first part, it is based on considering a differential equation as a
submanifold of a jet bundle and using the associated contact geometry via
the Vessiot spaces.  We showed in the second part of this work that
quasi-linear equations are special in the sense that their Vessiot
distribution is projectable.  This observation allows for relating the
geometric singularity analysis of fully non-linear implicit equations -- as
e.\,g.\ discussed by Arnold \cite{via:geoode} or Remizov
\cite{aor:poincare} -- with the more analytic approach to singularities of
quasi-linear equations -- as used e.\,g.\ by Rabier \cite{pjr:singular}.
We could show in Proposition \ref{prop:impasse} that not all impasse points
arising in the analysis of a quasi-linear equation stem from a singularity
of the equation.  Hence one indeed needs a special theory for the
quasi-linear case.

In the third part of this article, we gave a detailed geometric analysis of
a special class of second-order initial value problems, namely equations of
the form $g(x)u''=f(x,u,u')$ with initial data prescribed at a simple zero
$y$ of $g$.  The results represented a slight generalisation of those
obtained by Liang \cite{jfl:singivp} with completely different methods.  In
our opinion, our approach makes the appearance of a dichotomy or of a
resonance and the possible existence of solutions with only finite
regularity much more transparent.  All these effects follow immediately
from considering the phase portraits around proper impasse points at
different prolongation orders.  In particular, the analysis is almost
automatic and requires essentially no ingenuity.

The restriction to simple zeros of $g$ is crucial for this simplicity, as
it guarantees that the eigenvalue $\delta=g'(y)$ is always non-zero.
Hence, even in the border case $\gamma=0$ (which was ignored by Liang), we
never encountered a triple eigenvalue $0$.  The situation is quite
different for (generalised) Ginzburg-Landau equations of the form
$x^{2}u''+axu'+bu=f(u)$ for some function $f$ satisfying $f(0)=0$ as
studied by Ignat et al.~\cite{insz:unigl}.  If one analyses such an
equation via our approach, then one faces already in the first step the
problem of studying a stationary point of a three-dimensional system where
the Jacobian has a triple eigenvalue $0$.  If one succeeds here, for
instance by blowing up the initial point, then the remaining analysis
should be quite similar to the one presented here.

Another critical point in our approach is generally the question whether
one is able to determine the eigenvalues and -vectors of all required
Jacobians.  For the class studied here, we obtained triagonal matrices so
that this step was almost trivial posing only the problem of identifying
the relevant eigenvalue.  In a preliminary study of equations of the form
$g(u)u''=f(x,u,u')$ and $g(u')u''=f(x,u,u')$, i.\,e.\ of equations with a
truly non-linear left hand side, it turned out that this step becomes only
a little bit more difficult, but remains solvable at any prolongation
order.  We will present the findings for these two classes of equations
elsewhere.

We assumed throughout this article that we work with smooth functions (only
in Proposition \ref{prop:geosol} and in Remark \ref{rem:hoelder} we
considered the analytic case).  In fact, almost all of our results remain
true with only minor modifications, if we assume that $F$ or $f$ and $g$,
respectively, are only in $\mathcal{C}^{r}$ for some $r\geq2$.  Obviously,
we can now consider prolongations only up to order $r$ and also solutions
can only be guaranteed to be in class $\mathcal{C}^{r}$.  Thus we simply
must replace smooth by $\mathcal{C}^{r}$.

The situation is slightly more complicated for Theorems \ref{thm:Lnores}
and \ref{thm:Lres}.  For the interesting results about finite regularity,
we must assume that $r\geq k$ -- in fact we should have $r>k$.  In the
proofs we needed that the homeomorphism in the Hartman-Grobman theorem is
$\mathcal{C}^{1}$ and used the corresponding statement in \cite{ghr:dhgl}
which, however, requires smoothness of the considered vector field.
Hartman \cite{ph:lochom} showed already much earlier that the homeomorphism
is $\mathcal{C}^{1}$ under weaker conditions, namely when the linear part
defines a contraction and the non-linear part has uniformly Lipschitz
continuous partial derivatives.  We applied the Hartman-Grobman theorem to
the vector field $Y^{(k)}$ around the stationary point $\rho_{k}$.
According to \eqref{eq:Y1}, the field $Y^{(k)}$ depends on $h_{k+1}$ which
by \eqref{eq:hq} is obtained by differentiating $k-1$ times $f$.  If $r>k$,
then $h_{k+1}$ is still at least of class $\mathcal{C}^{2}$ and thus the
coefficients of $Y^{(k)}$ possess the required regularity.  For the
contraction property, we recall that the Jacobian of the reduced dynamics
has the eigenvalues $0<\gamma-(k-1)\delta<\delta$ and thus defines a
contraction, if and only if $\delta<1$.  We may perform a rescaling of the
independent variable $x\mapsto\alpha x$.  Then a simple computation shows
that the parameters $\delta$ and $\gamma$ rescale according to
$\delta\mapsto\delta/\alpha$ and $\gamma\mapsto\gamma/\alpha$.  Thus the
resonance condition is not affected by the rescaling, but we may assume
without loss of generality that $\delta<1$ so that Hartman's result can be
applied.

A referee pointed out that the resonance condition in the here considered
second-order initial value problem could be derived for analytic equations
via the Newton-Puiseux construction of Cano \cite{jc:puiseux}.  This
observation represents an interesting question for future research.  For
the class of problems considered here, it was straightforward to derive the
conditions both for the dichotomy of the existence theory and for the
resonances.  For other initial value problems, this is no longer the case
and the combination of our geometric techniques with such algebraic
approaches may prove very useful here.

\section*{Acknowledgements}

It is a pleasure for us to thank Sebastian Walcher (RWTH Aachen) and Peter
de Maesschalck (Hasselt University) for helpful discussions about the
theory of dynamical systems.  Oscar Saleta Reig (Universitat Aut\`onoma de
Barcelona) helped us with using P4.  We also appreciate the useful comments
of an anonymous referee.  This work was supported by the bilateral project
ANR-17-CE40-0036 and DFG-391322026 SYMBIONT.

\bibliographystyle{elsarticle-num}
\bibliography{QuasiLin}

\end{document}